\documentclass[reqno,english]{amsart}
\usepackage{amsfonts,amsmath,latexsym,verbatim,amscd,mathrsfs,color,array}
\usepackage[colorlinks=true]{hyperref}
\usepackage{amsmath,amssymb,amsthm,amsfonts,graphicx,color}
\usepackage{amssymb}
\usepackage{pdfsync}
\usepackage{epstopdf}
\usepackage{cite}

\newcommand{ \R} { \mathbb{R}}
\usepackage{graphicx}

\newtheorem{theorem}{Theorem}[section]
\newtheorem{lemma}[theorem]{Lemma}
\theoremstyle{remark}
\newtheorem{remark}[theorem]{Remark}
\newtheorem{proposition}[theorem]{Proposition}

\numberwithin{equation}{section}

\title[Classification of Blow-ups for Half Laplacian]{Classification of Blow-ups and Monotonicity Formula for Half Laplacian Nonlinear Heat Equation}

\author[B. Deng]{Bin Deng}
\address{\noindent School of Mathematical Sciences, University of Science and Technology of China, Hefei, Anhui Province, P.R. China, 230026}
\email{bingomat@mail.ustc.edu.cn}
\author[Y. Sire]{Yannick Sire}
\address{\noindent Department of Mathematics, Johns Hopkins University, 404 Krieger Hall, 3400 N. Charles Street, Baltimore, MD 21218, USA}
\email{sire@math.jhu.edu}
\author[J. Wei]{Juncheng Wei}
\address{\noindent Department of Mathematics, University of British Columbia, Vancouver, B.C., Canada, V6T 1Z2}
\email{jcwei@math.ubc.ca}
\author[K.Wu]{Ke Wu}
\address{\noindent School of Mathematics and Statistics, Xian Jiaotong University, Xian, Shanxi Province, P.R. China, 710049}
\email{wuke@stu.xjtu.edu.cn}

\begin{document}
\begin{abstract}
We consider the nonlinear half Laplacian heat equation
$$
    u_t+(-\Delta)^{\frac{1}{2}} u-|u|^{p-1}u=0,\quad \mathbb{R}^n\times (-T, 0).
    $$
    We prove that all blows-up are type I, provided that $n \leq 4$  and $ 1<p<p_{*} (n)$ where $ p_{*} (n)$ is an explicit exponent which is below $\frac{n+1}{n-1}$, the critical Sobolev  exponent. Central to our proof is a Giga-Kohn type monotonicity formula for half Laplacian and a Liouville type theorem for self-similar nonlinear heat equation. This is the first instance of a monotonicity formula at the level of the nonlocal equation, without invoking the extension to the half-space.

\end{abstract}
\maketitle

\tableofcontents

  \section{Introduction}
  In a series of seminal papers, Giga and Kohn \cite{GK, GK1, GK3} studied the asymptotic behavior of blow-up solutions to nonlinear heat equations with subcritical power nonlinearity:
  \begin{equation}\label{GK0}
\left\{\begin{array}{l}
u_t-\Delta u = |u|^{p-1} u, \ (x,t) \in \R^n\times(0,T)\\
u(x, 0)= u_0 (x)
\end{array}
\right.
\end{equation}
where $ 1<p <\frac{n+2}{n-2}$ for $ n\geq 3$ and $1< p<+\infty$ when $n=1,2$. We recall that the finite time blow up is said to be of type I if
$$ \limsup_{t\to T} (T-t)^{\frac{1}{p-1}}\|u(\cdot,t)\|_\infty<+\infty, $$
and of type II if
$$ \limsup_{t\to T} (T-t)^{\frac{1}{p-1}}\|u(\cdot,t)\|_\infty=+\infty,$$
where $T$  is the maximal existence time of the $L^{\infty}$ solution $u$.

In \cite{GK}, Giga-Kohn considered the equation
\begin{equation}\label{selfsimilar}
w_s-\Delta w + \frac{1}{2} y \cdot \nabla w + \frac{1}{p-1} w - |w|^{p-1} w=0,
\end{equation}
obtained from \eqref{GK0} by setting
$$w(y,s)=(-t)^{\frac{1}{p-1}}u(x,t),\quad x=(-t)^{\frac{1}{2}}y.$$
They proved that all bounded  global stationary solutions to \eqref{selfsimilar} are constants.  Then in \cite{GK1}, Giga and Kohn proved that all blow-ups of (\ref{GK0}) are Type I. In \cite{GK3}, Giga and Kohn showed that one can tell whether or not a point is a blow up point by examining the asymptotic behavior of a solution in a backward spacetime parabola. Moreover, they can give a local lower bound on the blow up rate. In \cite{MZ1} and \cite{MZ3}, Merle and Zaag classified all the bounded global nonnegative solutions to \eqref{selfsimilar} defined on $\mathbb{R}^n\times\mathbb{R}$.

Two central ingredients in Giga-Kohn's proof are: (1) a monotonicity formula with Gaussian weight for solutions of (\ref{selfsimilar}); (2) a weighted Pohozaev identity applied to steady states of (\ref{selfsimilar}). After these celebrated works, there have been many refined estimates, simplifications and applications. We refer to the papers \cite{FK, MZ1, MZ2, MZ3, PQS, S1} and the   book by Quittner and Souplet \cite{QS} for an up-to-date state of the art.

In this paper we initiate the attempt to generalize the Giga-Kohn program in the nonlocal setting. More precisely  we consider the following nonlinear half heat equation
  \begin{eqnarray}\label{halfHE}
    u_t+(-\Delta)^{\frac{1}{2}} u-|u|^{p-1}u=0,\quad \mathbb{R}^n\times (-T,
    0),
  \end{eqnarray}
  where $u$ is real-valued,  $p>1$, $0<T\leq\infty$ and $(-\Delta)^{-\frac{1}{2}}$ is  the half Laplacian.

In general, the fractional Laplacian $(-\Delta)^\alpha$, $\alpha\in(0,1)$,
  is defined in the following way,
  \begin{eqnarray}
    (-\Delta)^{\alpha} u(x):=c_{n,\alpha}P.V.
    \int_{\mathbb{R}^n}\frac{u(x)-u(x')}{|x-x'|^{n+2\alpha}}dx'.
  \end{eqnarray}
 The  normalizing constant is
  \begin{eqnarray}
    c_{n,\alpha}=\frac{2^{2\alpha}\Gamma(\frac{n+2\alpha}{2})}{\pi^{\frac{n}{2}}|\Gamma(-
    \alpha)|}
  \end{eqnarray}
  where $\Gamma(x)$ is the Gamma function. In our situation, $\alpha=\frac{1}{2}$, we denote
   \begin{eqnarray}\label{cn}
    c_n:=c_{n,\frac{1}{2}}=\frac{\Gamma(\frac{n+1}{2})}{\pi^{\frac{n+1}{2}}}.
  \end{eqnarray}
 Fractional Laplacian can also be defined as a pseudo-differential operator
  \begin{eqnarray*}
    \mathcal{F}( (-\Delta)^{\frac{1}{2}} u)(\xi)=|\xi|\mathcal{F}(u)(\xi)
  \end{eqnarray*}
  where $\mathcal{F}$ is defined by
  \begin{eqnarray*}
    \mathcal{F}(u)(\xi):=
    \int e^{-i x\cdot\xi}u(x)dx
  \end{eqnarray*}
  with $i$ the imaginary unit.

The kernel of the half heat equation
  \begin{eqnarray}
    u_t+(-\Delta)^{\frac{1}{2}}u=0
  \end{eqnarray}
  has an explicit expression (see e.g. \cite{CR})
  \begin{eqnarray}\label{halfheatkernel}
   P(x,t):=\mathcal{F}^{-1}( e^{- t|\xi|})=
    \frac{b_nt}{(t^2+|x|^2)^{\frac{n+1}{2}}},
  \end{eqnarray}
where
  \begin{eqnarray}
   b_n=\frac{\Gamma(\frac{n+1}{2})}{\pi^{\frac{n+1}{2}}}=\Big(\int
   \frac{ dx}{(1+|x|^2)^{\frac{n+1}{2}}}\Big)^{-1}.
   \end{eqnarray}
  We denote
  \begin{eqnarray}\label{explicit}
  \rho(x)=\frac{1}{b_n}P(x,1)=\frac{1}{(1+|x|^2)^{\frac{n+1}{2}}},
  \end{eqnarray}
    then
  \begin{eqnarray}\label{kernel}
    (-\Delta)^{\frac{1}{2}} \rho =n\rho +x\cdot\nabla\rho.
  \end{eqnarray}
Moreover, we have the following pointwise equality:
  \begin{eqnarray}\label{fradot}
    (-\Delta)^{\frac{1}{2}}(x\cdot\nabla\rho)=(-\Delta)^{\frac{1}{2}}\rho+
    x\cdot\nabla(-\Delta)^{\frac{1}{2}}\rho.
  \end{eqnarray}
 By \eqref{kernel} and \eqref{fradot}, we can get that
  \begin{eqnarray}\label{kerfradot}
    (-\Delta)^{\frac{1}{2}}(x\cdot\nabla\rho)&=&
    n\rho+(1+n)x\cdot\nabla\rho
    +x\cdot\nabla(x\cdot\nabla
    \rho).
  \end{eqnarray}
Furthermore, it is easy to see that
  \begin{eqnarray}
    x\cdot\nabla\rho(x)=-(n+1)\frac{|x|^2\rho}{1+|x|^2}\leq0.
  \end{eqnarray}
  For more results about fractional heat kernel, we refer to \cite{BG} and \cite{K}.

  In order to introduce our results, we define the quantity
\begin{eqnarray}\label{Mn}
 \begin{aligned}
  M_n
  :=&\ \sup_{y\in\mathbb{R}^n}\{\frac{1}{\rho(y)}
  \int_{\Omega_y}\frac{(\rho(y')-\rho(y))^2}{|y'-y|^{n+1}}\frac{1}{\rho(y')}dy'\}
  \\&\
  +\sup_{y\in\mathbb{R}^n}\{\frac{1}{\rho(y)^2}
  \int_{\mathbb{R}^n\setminus\Omega_y}\frac{(\rho(y')-\rho(y))^2}{|y'-y|^{n+1}}
  dy'\},
 \end{aligned}
\end{eqnarray}
where
\begin{eqnarray*}
  \Omega_y=B_{|y|}(0)=\{y'\in\mathbb{R}^n\ : \ |y'|<|y|\}.
\end{eqnarray*}
 We also define the exponent
\begin{equation}
\label{pn1}
p_*(n)=\frac{n+1-\frac{c_nM_n}{4}}{n-1+\frac{c_nM_n}{4}}
\end{equation}
where $c_n$ is defined by \eqref{cn}.
\begin{remark}
Some comments on the exponent $p_*(n)$ are in order. First a necessary condition for $p_*(n) >1$ is that $c_nM_n <4$. Some numerical computations show that
\begin{eqnarray*}
  c_2M_2\approx2.1498,\ \ c_3M_3\approx2.8406,\\
  c_4M_4\approx3.5561,\ \  c_5M_5\approx4.2839.
\end{eqnarray*}
As a consequence, our results hold {\sl only} for $n \leq 4$. Furthermore, it is easy to see that $p_*(n) < \frac{n+1}{n-1}$ for $n \geq 2$.
\end{remark}
\subsubsection*{Main results}
We first give a classification of backward \emph{self-similar} solutions. A \emph{self-similar solution} of (\ref{halfHE}) is of the form $$ u(x,t)= (-t)^{-\beta} w(\frac{x}{-t}),\quad\beta=\frac{1}{p-1},$$
   where $w$ satisfies
  \begin{eqnarray}\label{stableW}
    (-\Delta)^{\frac{1}{2}}w+y\cdot\nabla w+\beta w-|w|^{p-1}w=0, \ \text{in}\ \R^n.
  \end{eqnarray}
  It is easy to see that the only trivial solutions of \eqref{stableW} are $w\equiv0$ and $w\equiv \pm\beta^\beta$.

  Our first result is a classification of self-similar solution  for the semilinear  equation \eqref{halfHE}, which generalizes Theorem $1'$ in \cite{GK}.
\begin{theorem}\label{classifyselfsim}
  Let $n\leq4, 1< p\leq p_*(n)$ and let $u$ be a self-similar solution of (\ref{halfHE})
  satisfying the estimate
  \begin{eqnarray}
    \sup_{\mathbb{R}^n\times(-T,0)}(-t)^\beta|u(x,t)|<\infty,
  \end{eqnarray}
   then $u\equiv0$ or $u\equiv\pm \beta^\beta(-t)^{-\beta}$.
 \end{theorem}
Our next goal is to characterize the asymptotic behavior of finite time blow up solutions near a point, assuming suitable conditions.

  If $u$ is a solution of the half heat equation
  (\ref{halfHE}), then so do the rescaled functions
  \begin{eqnarray}
    u_{\lambda}(x,t):=\lambda^{\beta}u(\lambda x,\lambda t),\quad
    \beta=\frac{1}{p-1},
  \end{eqnarray}
  for each $\lambda>0$. In order to analyze the asymptotic behavior, we introduce the following backward self-similar transformation
  \begin{eqnarray}
    w(y,s):=(-t)^\beta u(x,t),
  \end{eqnarray}
  \begin{eqnarray}
    x=(-t)y,\ \ t=-e^{-s}.
  \end{eqnarray}
Then $w(y,s)$ satisfies the equation
\begin{eqnarray}\label{halfWE}
    w_s+(-\Delta)^{\frac{1}{2}} w+y\cdot\nabla w+\beta w-|w|^{p-1}w=0.
\end{eqnarray}
The following theorem, which generalizes results of \cite{GK}, classifies the backward self-similar heat equation (\ref{halfWE}).
\begin{theorem}\label{classifytypeI}
Let $n \leq 4, 1< p< p_*(n)$ and let $u$ be a solution of (\ref{halfHE})
  satisfying the  estimate
  \begin{eqnarray}
    \sup_{\mathbb{R}^n\times(-T, 0)}(-t)^\beta|u(x,t)|< \infty.
  \end{eqnarray}
We also assume the gradient of $u$ satisfies  the decay condition: fix $\delta>0$, for any $-T<t''<t'<0$,  there exists a constant $C(t',t'')< \infty$ such that
   \begin{eqnarray}
   |\nabla u(x,t)|\leq \frac{C(t',t'')}{1+|x|^\delta},\ (x,t)\in\R^n\times[t'',t'].
  \end{eqnarray}
 Then
  \begin{eqnarray}
  \label{limit1}
    \lim_{\lambda\rightarrow0}(-t)^{\beta}u_\lambda(x,t)=0\ \text{or}\ \pm\beta^\beta.
  \end{eqnarray}
  For each $c>0$, the limit \eqref{limit1} exists uniformly for any $|x|\leq c(-t)$.
\end{theorem}
Next, we could also obtain a Liouville-type theorem for ancient solutions of the equation \eqref{halfHE}. Usually, a solution of \eqref{halfHE} is called an \emph{ancient solution} if it exists for all time $t\in(-\infty, 0)$.
\begin{theorem}\label{halfliouville}
Let $n \leq 4, 1< p< p_*(n)$ and let $u$ be an ancient solution of (\ref{halfHE})
  satisfying
  \begin{eqnarray}
    \sup_{\mathbb{R}^n\times(-\infty,0)}(-t)^\beta|u(x,t)|< \infty.
  \end{eqnarray}
 We also assume the gradient of $u$ satisfies the decay condition: fix $\delta>0$, for any $-\infty<t''<t'<0$, there exists  a constant $C(t',t'')<\infty$ such that
   \begin{eqnarray}
   |\nabla u(x,t)|\leq \frac{C(t',t'')}{1+|x|^\delta},\ (x,t)\in\R^n\times[t'',t'].
  \end{eqnarray}
If
  \begin{eqnarray}\label{nondegenerate}
    \limsup_{t\rightarrow0}(-t)^\beta|u(0,t)|>0,
  \end{eqnarray}
  then
  \begin{eqnarray}
    u(x,t)=\pm\beta^\beta(-t)^{-\beta}.
  \end{eqnarray}
\end{theorem}
Our next goal is the growth rate estimate for the equation \eqref{halfHE}.
\begin{theorem}\label{classifyall}
 Let $n \leq 4, T<\infty$ and  $u$ be a solution of (\ref{halfHE})
  satisfying: fix $\delta>0$,  for any $-T<t'<0$,  there exists a constant $C(t')<\infty$ such that
  \begin{eqnarray}
    |u(x,t)|+|\nabla u(x,t)|(1+|x|^\delta)
 \leq C(t'),\ (x,t)\in\R^n\times[-T, t'].
  \end{eqnarray}
  If
 \begin{eqnarray}\label{p1}
   1< p<p_*(n),\ u\geq0,
 \end{eqnarray}
or
  \begin{eqnarray}\label{p2}
   1< p<\min\{1+\frac{2}{n}, p_*(n)\},
 \end{eqnarray}
 then
  \begin{eqnarray}\label{blowrate}
    \sup_{\mathbb{R}^n\times(-T,0)}(-t)^\beta|u(x,t)|<\infty.
  \end{eqnarray}
Furthermore,
\begin{eqnarray}\label{limit2}
     \lim_{\lambda\rightarrow0}(-t)^{\beta}u_\lambda(x,t)=0\ \text{or}\ \pm\beta^\beta.
  \end{eqnarray}
   For each $c>0$, the limit \eqref{limit2} exists uniformly for any $|x|\leq c(-t)$.
\end{theorem}
We point out that the assumption made on the gradient  is \emph{only} used to justify our computations. But this assumption can be verified if we consider suitable Cauchy problems. More precisely, we  consider the equation:
\begin{eqnarray}\label{CP}
\left\{\begin{array}{l}
u_t+(-\Delta)^{\frac{1}{2}} u = |u|^{p-1} u, \ (x,t) \in \R^n\times(0,T)\\
u(x, 0)= u_0 (x), \ x\in\R^n
\end{array}
\right.
\end{eqnarray}
where $T$ is the finite blow up time in the sense of
\begin{eqnarray}
T:=\sup\big\{t>0 : \sup_{(x,t)\in \R^n\times(0,t)}|u(x,t)|<\infty\big\}.
\end{eqnarray}
\begin{remark}
Let $1<p<1+\frac{1}{n}$ and let  $u_{0}$  be  a nontrivial, nonnegative and continuous function, then the nonegative solution of the equation \eqref{CP} blows up at some finite time (see \cite{S}). For  more local well-posedness of the Cauchy problem, we refer to \cite{EK}, \cite{FiK} and \cite{GuK}.
\end{remark}
\begin{theorem}
\label{classifyCP}
Let $u_0$ be a nontrivial($\not\equiv 0$), nonnegative and bounded continuous function satisfying
\begin{eqnarray}
\label{u0decay}
|\nabla u_0|(x)\leq \frac{C}{1+|x|^{\delta}},
\end{eqnarray}
for some $\delta>0$.
Let $n \leq 4, 1< p< p_*(n)$ and  let $u$ be a finite time blow up solution of the Cauchy problem (\ref{CP}),
then
 \begin{eqnarray}\label{3blowup}
    \lim_{t\rightarrow T}(T-t)^{\beta}u(x+y(T-t),t)=0\ \text{or}\ \pm\beta^\beta
  \end{eqnarray}
 uniformly for $y$ bounded, where $T$  is the maximal existence time of the $L^{\infty}$ solution $u$.
\end{theorem}
\medskip

\noindent
{\bf Main difficulties and ideas}: as mentioned before, Giga-Kohn's proof relies on two ingredients: first there is the {\em generalized Pohozaev} identity for self-similar solutions of (\ref{GK0})
\begin{equation} (\frac{n}{p+1} +\frac{2-n}{2}) \int |\nabla w|^2 \rho dy +\frac{1}{2} (\frac{1}{2}-\frac{1}{p+1}) \int |y|^2 |\nabla w|^2 \rho dy =0
\end{equation}
where $ \rho = e^{-\frac{1}{4} |y|^2}$ is the Gaussian. Second the following  Giga-Kohn energy functional
$$ E[w](s)= \frac{1}{2} \int |\nabla w|^2 \rho dy +\frac{1}{2} \beta \int |w|^2 \rho dy - \frac{1}{p+1} \int |w|^{p+1} \rho dy$$
is monotonically decreasing for backward self-similar nonlinear parabolic equation. The proof of both facts depend on some cancellations which seem only to work for the Laplace operator. Furthermore, the Gaussian weight $\rho$ ensures that all the computations are well-defined.

In our case, even with the explicit form (\ref{explicit}),  we are unable to obtain neither a monotonicity formula nor Pohozaev identity for full range $p <\frac{n+1}{n-1}$.  Furthermore, the weight \eqref{explicit} being polynomially decaying only, does not prevent our computations to be well-defined unless one assumes some {\sl a priori} decay on the solutions. This latter seems artificial but, even for the linear half heat equation, weak/strong solutions have always at most polynomial decay and this is optimal as proven in \cite{BSV}.

Instead we make use of some special integral decay in the dimension $n$ and we are able to prove a modified Pohozaev identity and monotonicity formula for partial range $ 1<p<p_{*} (n)$. See Propositions \ref{pohozaev} and \ref{monotonicity} below.  As far as we know this seems to be the first kind of monotonicity formula for nonlinear fractional heat equation at the level of the nonlocal operator.

In this paper we concentrate on half heat equations. The advantage is that the kernel is explicit and hence all the computations can be made explicitly. It may be possible to generalize to general $\alpha-$Laplacian heat equations if one knows the explicit formula for the kernel.
Indeed, let $\rho_\alpha(x)=\mathcal{F}^{-1}(e^{-|\xi|^{2\alpha}})$, the profile of the fractional heat kernel, and let $w$ be a solution of
\begin{eqnarray}\label{alphaW}
    (-\Delta)^{\frac{\alpha}{2}}w+\frac{1}{2\alpha}y\cdot\nabla w+\beta w-|w|^{p-1}w=0, \ \text{in}\ \R^n,
  \end{eqnarray}
  similar to \eqref{cominq}, we can obtain an inequality
  \begin{eqnarray}\label{alphainq}
 \begin{aligned}
   0\geq&\
  \Big(\frac{8\alpha-c_{n,\alpha}M_{n,\alpha}}{4}-\frac{(p-1)n}{(p+1)}\Big)\frac{c_{n,\alpha}}{4}
  \iint\frac{(w(y')-w(y))^2}{|y'-y|^{n+2\alpha}}\rho_\alpha(y')dy'dy
\\&\
  -\frac{(p-1)}{(p+1)}\frac{c_{n,\alpha}}{4}
  \iint\frac{(w(y')-w(y))^2}{|y'-y|^{n+2\alpha}}(y'\cdot\nabla\rho_\alpha)dy'dy,
 \end{aligned}
\end{eqnarray}
where
\begin{eqnarray}\label{Mna}
 \begin{aligned}
  M_{n,\alpha}
  :=&\ \sup_{y\in\mathbb{R}^n}\{\frac{1}{\rho_\alpha(y)}
  \int_{B_{|y|}}\frac{(\rho_\alpha(y')-\rho_\alpha(y))^2}{|y'-y|^{n+2\alpha}}\frac{1}{\rho_\alpha(y')}dy'\}
  \\&\
  +\sup_{y\in\mathbb{R}^n}\{\frac{1}{\rho_\alpha(y)^2}
  \int_{\mathbb{R}^n\setminus B_{|y|}}\frac{(\rho_\alpha(y')-\rho_\alpha(y))^2}{|y'-y|^{n+2\alpha}}
  dy'\}.
 \end{aligned}
\end{eqnarray}
If  $c_{n,\alpha}M_{n,\alpha}<8\alpha$, one can prove the corresponding results to general $\alpha$-Laplacian heat equation. Since we don't know the explicit formula of $\rho_\alpha$, it is very hard to compute the value of $M_{n,\alpha}$.

\section{Preliminaries: some regularity estimates}
In this section we collect some preliminary regularity estimates for half heat equation which will be useful in subsequent sections.
\begin{proposition}\label{propboundu}
    Let $0<T<\infty$ and let $u(x,t)$ be a solution of (\ref{halfHE})
    satisfying
    \begin{eqnarray}\label{boundu1}
      (-t)^\beta|u(x,t)|\leq M,\ (x,t)\in \mathbb{R}^n\times(-T,0),
    \end{eqnarray}
  then
    \begin{eqnarray}\label{boundu2}
      \sup_{\mathbb{R}^n\times(-r,0)}(-t)^{\beta+m}|\nabla^mu(x,t)|\leq C,\quad m=1,2,3,
    \end{eqnarray}
  for any $0<r<T$, with $C$ depending only on $n, p, r, T$ and $M$. If $T=\infty$ and if  \eqref{boundu1} holds on $\R^n\times(-\infty,0)$, then \eqref{boundu2} is valid on all of $\R^n\times(-\infty,0)$ and the constant $C$ only depends on $n, p$ and $M$.
  \end{proposition}
  \begin{proof}
   This follows from scaling arguments. First, we may assume $T=1$ and $\frac{6}{7}<r<r'=\frac{1+r}{2}<1$. For any $x_0\in\mathbb{R}^n$, we consider
    \begin{eqnarray}
      \tilde{u}(x,t)=u(x+x_0,t),
    \end{eqnarray}
    which still satisfies the half heat equation (\ref{halfHE}).
    The assumption (\ref{boundu1}) assures that
    \begin{eqnarray}
    \label{bound3}
      \sup_{\mathbb{R}^n\times(-1,-\frac{1}{2})}|\tilde{u}(x,t)|^p\leq 2^{\beta p}M^p.
    \end{eqnarray}
     The interior H\"older continuity
    for the half heat equation
    (see Theorem 1.3 in \cite{FR} or references \cite{JX, FKa}) yields,
    \begin{eqnarray}
      \|\tilde{u}\|_{C_t^{\frac{1}{2}}(B_2\times(-r',-\frac{2}{3})}
      +\|\tilde{u}\|_{C_x^{\frac{1}{2}}(B_2\times(-r',-\frac{2}{3}))}
      \leq C,
    \end{eqnarray}
    for some constant $C$ depending only on $n, p, r$ and $M$. Here we follow the
     notations in \cite{FR}, the H\"older seminorms are
     \begin{eqnarray}
       [u]_{C_x^{\frac{1}{2}}(\Omega\times I)}:=
       \sup_{\begin{subarray}{1}
       (x,t)\in\Omega\times I\\
       (x',t)\in\Omega\times I\end{subarray}}
       \frac{|u(x,t)-u(x',t)|}{|x-x'|^{\frac{1}{2}}}
     \end{eqnarray}
     and
     \begin{eqnarray}
       [u]_{C_t^{\frac{1}{2}}(\Omega\times I)}:=
       \sup_{\begin{subarray}{1}
       (x,t)\in\Omega\times I\\
       (x,t')\in\Omega\times I\end{subarray}}
       \frac{|u(x,t)-u(x,t')|}{|t-t'|^{\frac{1}{2}}}
     \end{eqnarray}
     for any $\Omega\times I\subset\mathbb{R}^{n+1}$.
     By the arbitrariness of
    $x_0$, we can obtain the H\"older regularities on the whole space for $u$ and
    $\tilde{u}$, say
    \begin{eqnarray}
       \|\tilde{u}\|_{C_t^{\frac{1}{2}}(\mathbb{R}^n\times(-r',-\frac{2}{3}))}
      +\|\tilde{u}\|_{C_x^{\frac{1}{2}}(\mathbb{R}^n\times(-r',-\frac{2}{3}))}
      \leq C.
    \end{eqnarray}
    Therefore, the interior Schauder estimates
    for the half heat equation
    (see  Theorem 1.2 in \cite{FR} or \cite{FK}) give,
    \begin{eqnarray}
         \|\tilde{u}\|_{C_t^{1+\frac{1}{2}}(B_1\times(-r,-\frac{3}{4}))}
      +\|\tilde{u}\|_{C_x^{1+\frac{1}{2}}(B_1\times(-r,-\frac{3}{4}))}
      \leq C,
    \end{eqnarray}
    for some constant $C$ depending only on $n, p, r$ and $M$. By an iteration argument, change $r$ if necessary, we obtain
     \begin{eqnarray}
         \|\tilde{u}\|_{C_t^{3+\frac{1}{2}}(B_1\times(-r,-\frac{3}{4}))}
      +\|\tilde{u}\|_{C_x^{3+\frac{1}{2}}(B_1\times(-r,-\frac{3}{4}))}
      \leq C.
    \end{eqnarray}
   By  arbitrariness of $x_0$,
    and $t_0<-1$, we conclude
    \begin{eqnarray}\label{boundu3}
       |\nabla u|+|\nabla^2u|+|\nabla^3u|
      \leq C,
    \end{eqnarray}
    for all $(x,t)\in\mathbb{R}^n\times(-r,-\frac{3}{4})$.

    Now we are going to prove (\ref{boundu2}) when $-\frac{3}{4}\leq t<0$. Fixing such $(x,t)\in \mathbb{R}^n
    \times [-\frac{3}{4},0)$, let $\lambda=-\frac{7}{6}t$ and consider
    \begin{eqnarray}
      v(z,\tau)=\lambda^\beta u(x+\lambda z,\lambda \tau).
    \end{eqnarray}
    It is easy to verifies that $v$ is well defined in $\mathbb{R}^n
    \times (-1,0)$, and again (\ref{boundu1}) assures that
    \begin{eqnarray}
     \sup_{\mathbb{R}^n\times(-1,0)} (-\tau)^\beta|v(z,\tau)|\leq M,
    \end{eqnarray}
    Applying (\ref{boundu3}) with $u$ replaced by $v$, and taking $z=0, \tau=
    -\frac{6}{7}\in(-r,-\frac{3}{4})$, we conclude that
     \begin{eqnarray}
     \label{bound4}
       \lambda^{\beta+1}|\nabla  u|+ \lambda^{\beta+2}|\nabla^2u|+
        \lambda^{\beta+3}|\nabla^3u|
      \leq C,\quad \text{at}\quad (x,t).
    \end{eqnarray}
    It is (\ref{boundu2}) since $\lambda=-\frac{7}{6}t$.

    For general $T>0$, consider the rescaled function $u_{\lambda}(x,t)=\lambda^{\beta}u(\lambda x,\lambda t)$ with $\lambda=T$. Then $u_\lambda$ satisfies the half heat equation \eqref{halfHE} on $\R^n\times(-1,0)$. Then it is the case considered above. The only difference is that we get a constant $C$ depending on $T$.  If the bound \eqref{boundu1} holds for $\R^n\times (-\infty,0)$, then the same argument yields the global version. Once again \eqref{boundu1} assures \eqref{bound3}, so \eqref{boundu3} yields \eqref{bound4} as above.
  \end{proof}
 Translating results of Proposition \ref{propboundu} to $w$ we have:
 \begin{proposition}\label{propboundw}
    Let $0<T<\infty$ and  let $w$ be a bounded solution of (\ref{halfWE}) in $\mathbb{R}^{n}\times(-\ln T,\infty)$.
    If $|w|\leq M$ for some positive constant $M$, then
      \begin{eqnarray}
       |\nabla w|+|\nabla^2w|+|\nabla^3w|
      \leq C,&\quad\label{boundw}\\
      |w_s+y\cdot\nabla w|+
      |\nabla(w_s+y\cdot\nabla w)|\leq C,&\label{boundw2}
    \end{eqnarray}
    for $(y,s)\in\mathbb{R}^n\times (-\ln r,\infty)$, here $0<r<T$ and  $C$ is a constant depending  only on $n, p, r, T$ and $M$. If $T=\infty$ and  $|w|\leq M$ globally, then \eqref{boundw} and \eqref{boundw2} are valid on all of $\R^{n+1}$, with $C$ only depending on $n, p$ and $M$.
  \end{proposition}
  \begin{proof}
    The estimates of $\nabla w, \nabla^2w$, and $\nabla^3w$  are merely
    restatements of (\ref{boundu2}).

    Since
    \begin{eqnarray}\label{boundhalfw}
     \begin{aligned}
      |(-\Delta)^{\frac{1}{2}}w(y)|
      =&\ \frac{c_n}{2}|\int\frac{(w(y+y')+w(y-y')-2w(y)}{|y'|^{n+1}}dy'|
    \\\leq&\
     4c_n|\nabla^2 w|_{L^\infty(B_1(y))}\int_{B_1(0)}\frac{dy'}{|y'|^{n-1}}
     \\&\
     +4c_n|w|_{L^\infty(\mathbb{R}^n\setminus B_1(0))}
     \int_{\mathbb{R}^n\setminus B_1(0)}\frac{dy'}{|y'|^{n+1}}
     \\\leq&\ C,
     \end{aligned}
    \end{eqnarray}
 it is not hard to get the estimate
    of $w_s+y\cdot\nabla w$ by means of the equation (\ref{halfWE}).

    Differentiating (\ref{halfWE}) with respect to $y_j$, $j=1,\cdots,n$,
    we can get that
    \begin{eqnarray}\label{halfgradw}
      \frac{\partial}{\partial y_j}(w_s+y\cdot\nabla w)
      +(-\Delta)^{\frac{1}{2}}\frac{\partial w}{\partial y_j}
      +\beta \frac{\partial w}{\partial y_j}
      -p|w|^{p-2}w\frac{\partial w}{\partial y_j}=0.
    \end{eqnarray}
    Using (\ref{boundhalfw}) with $w$ replaced by $\frac{\partial w}{\partial y_j}$,
     it is easy to get the estimate of $\frac{\partial}{\partial y_j}(w_s+y\cdot\nabla w)$
    by  the estimates (\ref{boundw}) and
     the equation (\ref{halfgradw}).
  \end{proof}
For the equation \eqref{stableW}, we can obtain the following result.
 \begin{proposition}
     Let $w$ be a bounded solution of (\ref{stableW}) in $\mathbb{R}^n$,
    with $|w|\leq M$. Then
      \begin{eqnarray}\label{boundw1}
      \begin{aligned}
       |\nabla w|+|\nabla^2w|+|\nabla^3w|
      &\leq C,\\
      |y\cdot\nabla w|+|\nabla(y\cdot\nabla w)|&\leq C,
      \end{aligned}
    \end{eqnarray}
    for all $y\in\mathbb{R}^n$, and constant $C$ depending only
    on $n,p$ and $M$.
  \end{proposition}
  \begin{proof}
    The estimates of $\nabla w, \nabla^2w, \nabla^3w, y\cdot\nabla w$  and
     $\nabla(y\cdot\nabla w)$ are merely
    restatements of (\ref{boundw}) and (\ref{boundw2}), since $w_s=0$ now.
  \end{proof}

  \section{Proof of Theorem \ref{classifyselfsim}: Classification of Self-similar solutions}\label{sec3}
In this section we prove a modified Pohozaev identity and prove Theorem \ref{classifyselfsim}. Following the line of ideas in \cite{GK}, we first obtain a modified Pohozaev-type identity (inequality). It is here where we first introduce the number $M_n$  at (\ref{Mn}). Observe that we do not use the Caffarelli-Silvestre extension \cite{CS}.
\begin{proposition}\label{pohozaev}
If $w(y)$ is a bounded solution of \eqref{stableW} in $\mathbb{R}^n$, for any $n\geq1$
and $p>1$, then
\begin{eqnarray}\label{comid}
 \begin{aligned}
  0=&\
  \Big(1-\frac{(p-1)n}{(p+1)}\Big)\frac{c_n}{4}
  \iint\frac{(w(y')-w(y))^2}{|y'-y|^{n+1}}\rho(y')dy'dy
\\&\
  -\frac{(p-1)}{(p+1)}\frac{c_n}{4}
  \iint\frac{(w(y')-w(y))^2}{|y'-y|^{n+1}}(y'\cdot\nabla\rho)dy'dy
    \\&\
    +\frac{c_n}{2}\iint\frac{(w(y')-w(y))(\rho(y')-\rho(y))}
   {|x-y|^{n+1}}(y\cdot\nabla w)dy'dy
  \\&\
    +\int(y\cdot\nabla w)^2\rho dy.
 \end{aligned}
\end{eqnarray}
As a consequence, we have
\begin{eqnarray}\label{cominq}
 \begin{aligned}
   0\geq&\
  \Big(\frac{4-c_nM_n}{4}-\frac{(p-1)n}{(p+1)}\Big)\frac{c_n}{4}
  \iint\frac{(w(y')-w(y))^2}{|y'-y|^{n+1}}\rho(y')dy'dy
\\&\
  -\frac{(p-1)}{(p+1)}\frac{c_n}{4}
  \iint\frac{(w(y')-w(y))^2}{|y'-y|^{n+1}}(y'\cdot\nabla\rho)dy'dy
 \end{aligned}
\end{eqnarray}
with $M_n$ defined by \eqref{Mn}.
\end{proposition}
\begin{proof}
   Let $w\in C^{2}(\mathbb{R}^n)$ be a function such that $\|w\|_{C^{2}(\mathbb{R}^n)}<\infty$  and let $\varphi(x)$ be a function such that $|\varphi(x)|\leq \frac{C}{(1+|x|^2)^\frac{n+1}{2}}$, we have the following identity
  \begin{eqnarray}\label{intid1}
   \begin{aligned}
   \int\varphi w (-\Delta)^{\frac{1}{2}} w dy=&\
    \frac{c_n}{2}\iint\frac{(w(y')-w(y))^2}{|y'-y|^{n+1}}\varphi(y')dy'dy\\
    &\ +\frac{1}{2}\int |w|^2(-\Delta)^{\frac{1}{2}} \varphi dy.
    \end{aligned}
  \end{eqnarray}
  Indeed, \eqref{boundhalfw} yields that $(-\Delta)^{\frac{1}{2}}w$ is bounded. Using the dominated convergence theorem and symmetrizing in $y$ and $y'$, $B_\delta=B_\delta(y)$, we have
  \begin{eqnarray*}\begin{aligned}
    &\int\varphi w (-\Delta)^{\frac{1}{2}} w dy\\
    =&\ c_n\lim_{\delta\to0}\int_{\R^n} \varphi(y) w(y)(\int_{\R^n\setminus B_\delta}\frac{w(y)-w(y')}{|y'-y|^{n+1}}dy') dy
    \\=&\
    \frac{c_n}{2}\lim_{\delta\to0}\iint_{(\R^n\setminus B_\delta)^2}\frac{(w(y)-w(y'))(\varphi(y) w(y)-\varphi(y') w(y'))}
     {|y'-y|^{n+1}}dy'dy
     \\&\ +\frac{c_n}{2}\lim_{\delta\to0}\int_{B_\delta}\varphi(y) w(y)\int_{\R^n\setminus B_\delta}\frac{w(y)-w(y')}{|y'-y|^{n+1}}dy'dy.\end{aligned}
  \end{eqnarray*}
  It is easy to see that
  \begin{eqnarray*}
  \begin{aligned}
    &|\int_{B_\delta}\varphi(y) w(y)\int_{\R^n\setminus B_\delta}\frac{w(y)-w(y')}{|y'-y|^{n+1}}dy'dy|\\
    &\leq C(\|w\|_{L^\infty}+\|\nabla^2 w\|_{L^\infty})\delta^n\to0
    \end{aligned}
  \end{eqnarray*}
  as  $\delta\to0$.  By the assumptions on $w$ and $\varphi$, we have
  \begin{eqnarray}
   \begin{aligned}
    \iint\frac{(w(y')-w(y))^2}{|y'-y|^{n+1}}|\varphi(y')|dy'dy&\leq
    C\iint_{B_1(y')}\frac{|\varphi(y')|}{|y'-y|^{n-1}}dydy'
    \nonumber\\&\
    +C\iint_{\mathbb{R}^n\setminus B_1(y')}\frac{|\varphi(y')|}{|y'-y|^{n+1}}dydy'
    \nonumber\\&\leq
    C\int|\varphi(y')|dy'<\infty.
    \end{aligned}
  \end{eqnarray}
  Therefore, using Fubini's theorem we obtain
   \begin{eqnarray*}\begin{aligned}
    &\int\varphi w (-\Delta)^{\frac{1}{2}} w dy\\
    =&\
    \frac{c_n}{2}\lim_{\delta\to0}\iint_{(\R^n\setminus B_\delta)^2}\frac{(w(y)-w(y'))^2}{|y'-y|^{n+1}}\varphi(y)dy'dy+
    \\&\
     \frac{c_n}{2}\lim_{\delta\to0}\iint_{(\R^n\setminus B_\delta)^2}\frac{(w(y)-w(y'))(\varphi(y) w(y')-\varphi(y') w(y'))}
     {|y'-y|^{n+1}}dy'dy
     \\=&\
    \frac{c_n}{2}\lim_{\delta\to0}\iint_{(\R^n\setminus B_\delta)^2}\frac{(w(y)-w(y'))^2}{|y'-y|^{n+1}}\varphi(y)dy'dy
    \\&\
     +\frac{c_n}{4}\lim_{\delta\to0}\iint_{(\R^n\setminus B_\delta)^2}\frac{(w(y)^2-w(y')^2)(\varphi(y)-\varphi(y'))}
     {|y'-y|^{n+1}}dy'dy
     \\=&\
    \frac{c_n}{2}\iint\frac{(w(y)-w(y'))^2}{|y'-y|^{n+1}}\varphi(y)dy'dy
     \\&\
     +\frac{c_n}{2}\iint\frac{\varphi(y)-\varphi(y')}
     {|y'-y|^{n+1}}w(y)^2dy'dy,\end{aligned}
  \end{eqnarray*}
  which is (\ref{intid1}).

  Multiplying  \eqref{stableW} by $\rho w$ and using (\ref{intid1}) with
   $\varphi$ replaced by $\rho$ and using integration by parts, we have
  \begin{eqnarray}\nonumber
   \begin{aligned}
    0=&\ \frac{c_n}{2}\iint\frac{(w(y')-w(y))^2}{|y'-y|^{n+1}}\rho(y')dy'dy+\frac{1}{2}
    \int |w|^2(-\Delta)^{\frac{1}{2}} \rho dy\\
    &\ -\frac{n}{2}\int |w|^2\rho dy-\frac{1}{2}\int|w|^2(y\cdot\nabla\rho) dy
    +\beta\int|w|^2\rho dy-\int|w|^{p+1}\rho dy.
    \end{aligned}
  \end{eqnarray}
  Using equation (\ref{kernel}), we have the first identity
  \begin{eqnarray}\label{id1}
   \begin{aligned}
     0=&\ \frac{c_n}{2}\iint\frac{(w(y')-w(y))^2}{|y'-y|^{n+1}}\rho(y')dy'dy
    +\beta\int|w|^2\rho dy
     -\int|w|^{p+1}\rho dy.
    \end{aligned}
  \end{eqnarray}
  Since $w$ is bounded, the estimate (\ref{boundw1}) implies that  $\|w\|_{C^{2}(\mathbb{R}^n)}<\infty$. In view of this fact and the observation that
  \begin{eqnarray}
    \frac{R}{2}\int_{\partial B_R}|w|^2\rho dy\leq\frac{C}{R}\rightarrow 0\
    \text{as}\  R\rightarrow\infty,
    \nonumber
  \end{eqnarray}
  the procedure of using integration by parts can be justified. Similarly, multiplying the equation \eqref{stableW} by $(y\cdot\nabla\rho)w$ and using (\ref{intid1}) with
   $\varphi$ replaced by $y\cdot\nabla \rho$ and using integration by parts, we have
  \begin{eqnarray*}
   \begin{aligned}
    0=&\ \frac{c_n}{2}\iint\frac{(w(y')-w(y))^2}{|y'-y|^{n+1}}(y'\cdot\nabla\rho)dy'dy
    +\frac{1}{2}\int |w|^2(-\Delta)^{\frac{1}{2}} (y\cdot\nabla\rho) dy\\
    &\ -\frac{n}{2}\int|w|^2(y\cdot\nabla\rho)dy
    -\frac{1}{2}\int|w|^2y\cdot\nabla
    (y\cdot\nabla\rho)dy
\\&\
  +\beta\int|w|^2(y\cdot\nabla\rho)dy
    -\int|w|^{p+1}(y\cdot\nabla\rho)dy.\end{aligned}
  \end{eqnarray*}
  Using the equation (\ref{kerfradot}), we obtain the second identity
  \begin{eqnarray}\label{id2}
    \begin{aligned}
    0=&\ \frac{c_n}{2}\iint\frac{(w(y')-w(y))^2}{|y'-y|^{n+1}}(y'\cdot\nabla\rho)dy'dy+
    \frac{n}{2}\int|w|^2\rho dy
    \\&\ +(\frac{1}{2}+\beta)\int|w|^2(y\cdot\nabla\rho)dy
    -\int|w|^{p+1}(y\cdot\nabla\rho)dy.\end{aligned}
  \end{eqnarray}
  Since
  \begin{eqnarray}
    \frac{R}{2}\int_{\partial B_R}|w|^2(y\cdot\nabla\rho) dy\leq\frac{C}{R}
    \rightarrow 0\
    \text{as}\  R\rightarrow\infty,
  \end{eqnarray}
the procedure of using integration by parts can also be justified.
To get the third identity, we define a quantity
  \begin{eqnarray}\label{quantityE}
   \begin{aligned}
    E[u]:=&\ \frac{c_n}{4}\iint\frac{(u(y')-u(y))^2}{|y'-y|^{n+1}}\rho(y')dy'dy+
    \frac{\beta}{2}\int|u|^2\rho dy
    \\&\
    -\frac{1}{p+1}\int|u|^{p+1}\rho dy.\end{aligned}
  \end{eqnarray}
 Let $w_\lambda(y)=w(\lambda y)$, then
  \begin{eqnarray}\label{leftlambdaE}
   \begin{aligned}
    \frac{dE[w_\lambda]}{d\lambda}|_{\lambda=1}=&\
  \frac{c_n}{2}\iint\frac{(w(y')-w(y))(y'\cdot\nabla w-y\cdot\nabla w)}
   {|y'-y|^{n+1}}\rho(y')dy'dy
\\&\
  +\beta\int\rho w(y\cdot\nabla w)dy
  -\int\rho |w|^{p-1}w(y\cdot \nabla w)dy
\\=&\
  -\frac{c_n}{2}\iint\frac{(w(y')-w(y))(\rho(y')-\rho(y))}
   {|y'-y|^{n+1}}(y\cdot\nabla w)dy'dy
\\&\
  +\int\rho(y\cdot\nabla w)(-\Delta)^{\frac{1}{2}} w dy
  +\beta\int\rho w(y\cdot\nabla w)dy
\\&\
  -\int\rho |w|^{p-1}w(y\cdot \nabla w)dy
\\=&\
   -\frac{c_n}{2}\iint\frac{(w(y')-w(y))(\rho(y')-\rho(y))}
   {|y'-y|^{n+1}}(y\cdot\nabla w)dy'dy
\\&\
  -\int(y\cdot\nabla w)^2\rho dy.
   \end{aligned}
  \end{eqnarray}
 By (\ref{boundw1}), we have
  \begin{eqnarray*}
   \begin{aligned}
    \iint\frac{(y'\cdot\nabla w-y\cdot\nabla w)^2}
    {|y'-y|^{n+1}}\rho(y')dxdy\leq&\
    C\iint_{B_1(y')}\frac{\rho(y')}{|y'-y|^{n-1}}dydy'
   \\&\
    +C\iint_{\mathbb{R}^n\setminus B_1(y')}\frac{\rho(y')}{|y'-y|^{n+1}}dydy'
    \nonumber\\\leq&\
    C\int\rho(y')dy'<\infty
    \end{aligned}
   \end{eqnarray*}
    and
    \begin{eqnarray*}
\int(y\cdot\nabla w)^2\rho dy<\infty.
  \end{eqnarray*}
  Therefore, the  procedure of differentiating $E[w_\lambda]$ with respect to $\lambda$
   can be justified. On the other hand,
  \begin{eqnarray*}
   \begin{aligned}
    E[w_\lambda]=&\ \frac{\lambda^{1-n}c_n}{4}\iint\frac{(w(y')-w(y))^2}{|y'-y|^{n+1}}
    \rho_\lambda(y')dy'dy+\frac{\beta\lambda^{-n}}{2}\int|w|^2\rho_\lambda dy
    \\&\
    -\frac{\lambda^{-n}}{p+1}\int|w|^{p+1}\rho_\lambda dy\end{aligned}
  \end{eqnarray*}
  with $\rho_\lambda(y)=\rho(\frac{y}{\lambda})$.
  Therefore,
  \begin{eqnarray}\label{diflambdaE2}
   \begin{aligned}
    \frac{dE[w_\lambda]}{d\lambda}|_{\lambda=1}=&\
    \frac{1-n}{4}c_n\iint\frac{(w(y')-w(y))^2}{|y'-y|^{n+1}}
    \rho(y')dy'dy
    \\&\
    -\frac{c_n}{4}\iint\frac{(w(y')-w(y))^2}{|y'-y|^{n+1}}
    (y\cdot\nabla\rho)dy'dy
    \\ &\
    -\frac{n\beta}{2}\int|w|^2\rho dy'
    -\frac{\beta}{2}\int|w|^2(y\cdot\nabla\rho) dy
    \\&\
    +\frac{n}{p+1}\int|w|^{p+1}\rho dy
    +\frac{1}{p+1}\int|w|^{p+1}(y\cdot\nabla\rho) dy.
    \end{aligned}
  \end{eqnarray}
  By (\ref{leftlambdaE}) and \eqref{diflambdaE2}, we  can obtain the third identity
   \begin{eqnarray}\label{id3}
    \begin{aligned}
    0=&\
    \frac{1-n}{4}c_n
    \iint\frac{(w(y')-w(y))^2}{|y'-y|^{n+1}}
    \rho(y')dy'dy
    \\&\
    -\frac{c_n}{4}\iint\frac{(w(y')-w(y))^2}{|y'-y|^{n+1}}
    (y'\cdot\nabla\rho)dy'dy
    \\ &\
    -\frac{n\beta}{2}\int|w|^2\rho dy'
    -\frac{\beta}{2}\int|w|^2(y\cdot\nabla\rho) dy
    \\&\
    +\frac{n}{p+1}\int|w|^{p+1}\rho dy
    +\frac{1}{p+1}\int|w|^{p+1}(y\cdot\nabla\rho) dy
    \\&\
     +\frac{c_n}{2}\iint\frac{(w(y')-w(y))(\rho(y')-\rho(y))}
   {|y'-y|^{n+1}}(y\cdot\nabla w)dy'dy
\\&\
    +\int(y\cdot\nabla w)^2\rho dy.
    \end{aligned}
  \end{eqnarray}
Combining the identities (\ref{id1}), (\ref{id2}) and (\ref{id3})
in the following way:
\begin{eqnarray*}
  \frac{n}{p+1}\cdot(\ref{id1})+\frac{1}{p+1}\cdot(\ref{id2})
  +1\cdot(\ref{id3}),
\end{eqnarray*}
we have derived \eqref{comid}.

Now we estimate the third term in the above Pohozaev's type identity (\ref{comid})  using Cauchy-Schwarz inequality and H\"older inequality,
\begin{eqnarray}
 \begin{aligned}\nonumber
  I:=&\
  |\int\big(\int\frac{(\rho(y')-\rho(y))(w(y')-w(y))}{|y'-y|^{n+1}}dy'\big)
  (y\cdot\nabla w)dy|
   \\  \leq&\
  \frac{\epsilon}{2}\int\big(
  \int\frac{(\rho(y')-\rho(y))(w(y')-w(y))}{|y'-y|^{n+1}}dy'\big)^2\frac{1}{\rho(y)}dy
   \\&\
  +\frac{1}{2\epsilon}\int\rho(y\cdot\nabla w)^2dy
   \\\leq&\
 \frac{\epsilon}{2}\int\big(
  \int_{\Omega_y}\frac{(\rho(y')-\rho(y))^2}{|y'-y|^{n+1}}\frac{1}{\rho(y')}dy'\big)\big(
  \int_{\Omega_y}\frac{(w(y')-w(y))^2}{|y'-y|^{n+1}}\rho(y')dy'\big)
  \frac{1}{\rho(y)}dy
  \\&\
  +\frac{\epsilon}{2}\int\big(
  \int_{\mathbb{R}^n\setminus\Omega_y}\frac{(\rho(y')-\rho(y))^2}{|y'-y|^{n+1}}dy'\big)
  \big(\int_{\mathbb{R}^n\setminus\Omega_y}\frac{(w(y')-w(y))^2}
  {|y'-y|^{n+1}}dy'\big)\frac{1}{\rho(y)}dy
   \\&\
  +\frac{1}{2\epsilon}\int\rho(y\cdot\nabla w)^2dy,
  \end{aligned}
\end{eqnarray}
here $\epsilon>0$ is a constant which will be determined later.
Let
\begin{eqnarray}\label{f12}
  \begin{aligned}
    f_1(y):=&\ \frac{1}{\rho(y)}
  \int_{\Omega_y}\frac{(\rho(y')-\rho(y))^2}{|y'-y|^{n+1}}\frac{1}{\rho(y')}dy',
  \\
    f_2(y):=&\ \frac{1}{\rho(y)^2}
  \int_{\mathbb{R}^n\setminus\Omega_y}\frac{(\rho(y')-\rho(y))^2}{|y'-y|^{n+1}}
  dy',
  \end{aligned}
\end{eqnarray}
and let
\begin{eqnarray}
  \Omega_y=B_{|y|}(0)=\{y'\in\mathbb{R}^n\ : \ |y'|<|y|\}.
\end{eqnarray}
We have
\begin{eqnarray}
 \begin{aligned}
  I \leq&\
 \frac{\epsilon}{2}\int f_1(y)\big(
  \int_{\Omega_y}\frac{(w(y')-w(y))^2}{|y'-y|^{n+1}}\rho(y')dy'\big)dy+\frac{1}{2\epsilon}\int\rho(y\cdot\nabla w)^2dy
  \\&\
  +\frac{\epsilon}{2}\int f_2(y)
  \big(\int_{\mathbb{R}^n\setminus\Omega_y}\frac{(w(y')-w(y))^2}
  {|y'-y|^{n+1}}\rho(y)dy'\big)dy
  \\ \leq&\
 \frac{\epsilon}{2}(\sup_{y\in\mathbb{R}^n}f_1(y))\iint
 \frac{(w(y')-w(y))^2}{|y'-y|^{n+1}}\rho(y')dy'dy +\frac{1}{2\epsilon}\int\rho(y\cdot\nabla w)^2dy
  \\&\
  +\frac{\epsilon}{2}(\sup_{y\in\mathbb{R}^n}f_2(y))\iint\frac{(w(y')-w(y))^2}
  {|y'-y|^{n+1}}\rho(y)dy'dy.
 \end{aligned}
\end{eqnarray}
Let
$$
  M_n=\sup_{y\in\mathbb{R}^n}f_1(y)+\sup_{y\in\mathbb{R}^n}f_2(y).
$$
Then, by symmetry of $y$ and $y'$, we get
\begin{eqnarray}\label{estOmega}
  I&\leq&
  \frac{M_n\epsilon}{2}
  \iint\frac{(w(y')-w(y))^2}{|y'-y|^{n+1}}\rho(y')dy'dy
  +\frac{1}{2\epsilon}\int\rho(y\cdot\nabla w)^2dy.
\end{eqnarray}
Selecting $\epsilon=\frac{c_n}{4}$ in (\ref{estOmega}) and plugging it into
(\ref{comid}) yields \eqref{cominq}.
\end{proof}
If $c_nM_n<4$ and $1<p\leq p_*(n)$,  the coefficients in the right hand side of \eqref{cominq} will become positive.
Therefore, we can obtain the following result.
\begin{theorem}\label{PW}
  Let $n \leq 4, 1< p \leq p_*(n)$ and let $w$ be a bounded solution of the equation (\ref{stableW}) in $\mathbb{R}^n$.  Then $w\equiv0$ or $w\equiv\pm \beta^\beta$.
\end{theorem}
\begin{proof}[Proof of Theorem \ref{classifyselfsim}]
It follows from \eqref{stableW} and Theorem \ref{PW}.
\end{proof}
\section{Monotonicity Formula and Proof of Theorems \ref{classifytypeI} and \ref{halfliouville}}
In this section we derive a modified Giga-Kohn monotonicity formula and prove  Theorems \ref{classifytypeI} and \ref{halfliouville}. Let $0<T\leq \infty$. We consider solutions $u$ of (\ref{halfHE})
which satisfy
\begin{eqnarray}
  |u(x,t)|\leq C(-t)^{-\beta},\quad
  (x,t)\in\mathbb{R}^n\times(-T,0)\label{boundu}
\end{eqnarray}
and the decay condition: fix  $\delta>0$, for any $-T<t''<t'<0$, there exists a constant $C(t',t'')<\infty$ such that
\begin{eqnarray}
\label{nablaudecay}
|\nabla u|(x,t)|\leq \frac{C(t',t'')}{1+|x|^\delta},\ (x,t)\in\R^n\times[t'', t'].
\end{eqnarray}

Let $u_{\lambda}(x,t)=\lambda^\beta u(\lambda x, \lambda t)$, then $u_{\lambda}(x, t)$
remain bounded independently of $\lambda$ away from $t=0$ since
\begin{eqnarray}
  |\lambda^\beta u(\lambda x, \lambda t)|<
  C\lambda^\beta(-\lambda t)^{-\beta}
  =C(-t)^{-\beta}.
\end{eqnarray}
By Proposition \ref{propboundu}, we can take weak limits
\begin{eqnarray}
  \lim_{\lambda_j\rightarrow0}u_{\lambda_j}=u_0,
  \quad
  \lim_{\lambda_j'\rightarrow\infty}u_{\lambda_j'}=u_\infty,
\end{eqnarray}
for suitable sequences $\lambda_j\rightarrow0$, $\lambda_j'\rightarrow\infty$.

Heuristically, the rescaled function $w(y,s)=(-t)^\beta u(x,t)$, with
$x=(-t)y$ and $t=-e^{-s}$, is a bounded solution of (\ref{halfWE}). By \eqref{nablaudecay}, $w$ also satisfies the decay condition: fix $\delta>0$, for any $-\ln T<s''<s'<\infty$, there exists a constant $C(s'',s')<\infty$ such that
\begin{eqnarray}
|y\cdot\nabla w(y,s)|\leq C(s'',s')\frac{|y|}{1+|y|^{\delta}},\ (y,s)\in\R^n\times[s'', s'].\label{boundygradw}
\end{eqnarray}
Then it is easy to get the estimate for $w_s$ by means of the equation (\ref{halfWE})
, i.e.,
\begin{eqnarray}\label{boundws}
 | w_s(y,s)|\leq C(s'',s')\frac{|y|}{1+|y|^{\delta}},\ (y,s)\in\R^n\times[s'', s'].
\end{eqnarray}
In the following, we use $w(y)$ instead of
$w(y,s)$ in all double integral for simplicity.

At the beginning, we define a quantity which plays the role of `energy' in our situation,
\begin{eqnarray}\label{quantityhatE}
  \begin{aligned}
  \hat{E}[w](s):=&\ \frac{c_n}{4}\iint\frac{(w(y')-w(y))^2}{|y'-y|^{n+1}}\rho(y')dy'dy+
    \frac{\beta}{2}\int|w|^2\rho dy
\\&\
  +\frac{1}{2(p+1)}\int |w|^2(-\Delta)^{\frac{1}{2}}\rho dy
  -\frac{1}{p+1}\int|w|^{p+1}\rho dy.\end{aligned}
\end{eqnarray}
Then we have the following monotonicity formula.
\begin{proposition}\label{monotonicity}
 Let $n\leq 4, -\ln T<a<b<\infty$ and let  $w$ be a bounded solution of (\ref{halfWE}) on $\mathbb{R}^n\times(-\ln T,\infty)$ satisfying
  \eqref{boundygradw}, then
  \begin{eqnarray}\label{energyinq}
   \begin{aligned}
    -\hat{E}[w]|_a^b=&\
    \hat{E}[w](a)-\hat{E}[w](b)
    \\\geq&\
       \big(1-\frac{c_n}{4\epsilon}\big)\int_a^b\int(w_s+y\cdot\nabla w)^2\rho dyds
   \\&\
   +d_{n,p,\epsilon}
   \int_a^b\iint\frac{(w(y')-w(y))^2}{|y'-y|^{n+1}}\rho(y')dy'dyds
  \\&\
  -d_{n,p}
  \int_a^b\iint\frac{(w(y')-w(y))^2}{|y'-y|^{n+1}}(y'\cdot\nabla\rho)dy'dyds,
   \end{aligned}
  \end{eqnarray}
 where $\epsilon\in[\frac{c_n}{4}, \frac{(p+1)-(p-1)n}{(p+1)M_n}]$, $d_{n,p,\epsilon}=(\frac{1-M_n\epsilon}{4}-\frac{(p-1)n}{4(p+1)})c_n\geq0$ and
  $d_{n,p}=\frac{(p-1)c_n}{4(p+1)}$.
\end{proposition}
\begin{proof}
 Similar to the proof in section \ref{sec3}, we first derive two identities.

Multiplying the equation ({\ref{halfWE}}) with $\rho w$ and using integration by parts, we have
\begin{eqnarray}\label{sid1}
  \begin{aligned}
  0=&\ \int w_sw\rho dy
  +\frac{c_n}{2}\iint\frac{(w(y')-w(y))^2}{|y'-y|^{n+1}}\rho(y')dy'dy
  \\&\
  +\beta\int|w|^2\rho dy-\int|w|^{p+1}\rho dy.
\end{aligned}
\end{eqnarray}
Multiplying the equation ({\ref{halfWE}}) with $(y\cdot\nabla\rho)w$
and using integration by parts, we have
\begin{eqnarray}\label{sid2}
 \begin{aligned}
  0=&\
  \int w_sw(y\cdot\nabla\rho)dy+
   \frac{c_n}{2}\iint\frac{(w(y')-w(y))^2}{|y'-y|^{n+1}}(y'\cdot\nabla\rho)dy'dy
\\&\ +
    \frac{n}{2}\int|w|^2\rho dy
   +(\frac{1}{2}+\beta)\int|w|^2(y\cdot\nabla\rho)dy
  \\&\   -\int|w|^{p+1}(y\cdot\nabla\rho)dy.\end{aligned}
\end{eqnarray}
 Thanks to  (\ref{boundw}), (\ref{boundygradw}) and (\ref{boundws}), the procedures of using integration by parts
are easy to justify for any $-\ln T<s<\infty$. We  point out that (\ref{boundygradw}) and (\ref{boundws}) are only used to ensure that  $\int w_sw\rho dy$ and  $\int (y\cdot\nabla w)w\rho dy$ are  well-defined.

In order to continue the proof, we define
  \begin{eqnarray}\label{quantityEs}
   \begin{aligned}
    E[w](s):=&\ \frac{c_n}{4}\iint\frac{(w(y')-w(y))^2}{|y'-y|^{n+1}}\rho(y')dy'dy+
    \frac{\beta}{2}\int|w|^2\rho dy
\\&\ -\frac{1}{p+1}\int|w|^{p+1}\rho dy.\end{aligned}
  \end{eqnarray}
If $w_\lambda(y,s)=w(\lambda y, s+\ln\lambda)$, then $\frac{dw_{\lambda}}{d\lambda}|_{
\lambda=1}=w_s+y\cdot\nabla w$. Similar to (\ref{leftlambdaE}), we have
\begin{eqnarray}
 \begin{aligned}
  \frac{dE[w_\lambda]}{d\lambda}|_{\lambda=1}
  =&\ -\int (w_s+y\cdot\nabla w)^2\rho dy
\\&\
  -\frac{c_n}{2}\iint\frac{(w(y')-w(y))(\rho(y')-\rho(y))}{|y'-y|^{n+1}}(w_s+
  y\cdot\nabla w)dy'dy
 \end{aligned}
\end{eqnarray}
On the other hand, similar to (\ref{diflambdaE2}),
\begin{eqnarray}
 \begin{aligned}
  \frac{dE[w_\lambda]}{d\lambda}|_{\lambda=1}
  =&\
  \frac{dE[w]}{ds}+
  \frac{1-n}{4}c_n\iint\frac{(w(y')-w(y))^2}{|y'-y|^{n+1}}
    \rho(y')dy'dy
    \\&\
    -\frac{c_n}{4}\iint\frac{(w(y')-w(y))^2}{|y'-y|^{n+1}}
    (y'\cdot\nabla\rho)dy'dy
 \\ &\
    -\frac{n\beta}{2}\int|w|^2\rho dy'
    -\frac{\beta}{2}\int|w|^2(y\cdot\nabla\rho) dy
     \\&\
    +\frac{n}{p+1}\int|w|^{p+1}\rho dy
    +\frac{1}{p+1}\int|w|^{p+1}(y\cdot\nabla\rho) dy.
  \end{aligned}
\end{eqnarray}
Therefore,
\begin{eqnarray}\label{sid3}
  \begin{aligned}
  -\frac{dE[w]}{ds}
  =&\
  \frac{1-n}{4}\iint\frac{(w(y')-w(y))^2}{|y'-y|^{n+1}}
    \rho(y')dy'dy
    \\&\
    -\frac{c_n}{4}\iint\frac{(w(y')-w(y))^2}{|y'-y|^{n+1}}
    (y'\cdot\nabla\rho)dy'dy
     \\ &\
    -\frac{n\beta}{2}\int|w|^2\rho dy'
    -\frac{\beta}{2}\int|w|^2(y\cdot\nabla\rho) dy
     \\&\
    +\frac{n}{p+1}\int|w|^{p+1}\rho dy
    +\frac{1}{p+1}\int|w|^{p+1}(y\cdot\nabla\rho) dy
     \\&\
    +\int (w_s+y\cdot\nabla w)^2\rho dy+ II,
 \end{aligned}
\end{eqnarray}
here
\begin{eqnarray}
  II:= \frac{c_n}{2}\iint\frac{(w(y')-w(y))(\rho(y')-\rho(y))}{|y'-y|^{n +1}}(w_s+
  y\cdot\nabla w)dy'dy.
\end{eqnarray}
Combining the identities (\ref{sid1}), (\ref{sid2}) and (\ref{sid3})
in the following way,
\begin{eqnarray*}
  \frac{n}{p+1}\cdot(\ref{sid1})+\frac{1}{p+1}\cdot(\ref{sid2})
  +1\cdot(\ref{sid3}),
\end{eqnarray*}
we have
\begin{eqnarray}\label{comsid}
 \begin{aligned}
  -\frac{dE[w]}{ds}=&\
  \frac{n}{p+1}\int w_sw\rho dy
  +\frac{1}{p+1}\int w_sw(y\cdot\nabla\rho) dy
\\&\
  +(\frac{1}{4}-\frac{(p-1)n}{4(p+1)})c_n
  \iint\frac{(w(y')-w(y))^2}{|y'-y|^{n+1}}\rho(y')dy'dy
 \\&\
  -\frac{(p-1)c_n}{4(p+1)}
  \iint\frac{(w(y')-w(y))^2}{|y'-y|^{n+1}}(y'\cdot\nabla\rho)dy'dy
  \\&\
    +\int (w_s+y\cdot\nabla w)^2\rho dy+ II.
 \end{aligned}
\end{eqnarray}
Estimating $II$ in the same way as to estimate $I$ in (\ref{estOmega}), and
using (\ref{kernel}), we have
\begin{eqnarray}\label{estEs}
 \begin{aligned}
  -\frac{dE[w]}{ds}\geq&\
  \frac{1}{2(p+1)}\frac{d}{ds}\int |w|^2(-\Delta)^{\frac{1}{2}}\rho dy
\\&\
  +(1-\frac{c_n}{4\epsilon})\int(w_s+y\cdot\nabla w)^2\rho dy
 \\&\
   +d_{n,p,\epsilon}
  \iint\frac{(w(y')-w(y))^2}{|y'-y|^{n+1}}\rho(y')dy'dy
   \\&\
  -d_{n,p}
  \iint\frac{(w(y')-w(y))^2}{|y'-y|^{n+1}}(y'\cdot\nabla\rho)dy'dy,
 \end{aligned}
\end{eqnarray}
here $d_{n,p,\epsilon}=(\frac{1-M_n\epsilon}{4}-\frac{(p-1)n}{4(p+1)})c_n$ and
 $d_{n,p}=\frac{(p-1)c_n}{4(p+1)}$.
By the definition of $\hat{E}[w]$ and (\ref{estEs}), we can get  the following monotonicity formula
\begin{eqnarray}\label{monohatE1}
\begin{aligned}
  -\frac{d\hat{E}[w]}{ds}\geq&\
  (1-\frac{c_n}{4\epsilon})\int(w_s+y\cdot\nabla w)^2\rho dy
\\&\
+d_{n,p,\epsilon}
  \iint\frac{(w(y')-w(y))^2}{|y'-y|^{n+1}}\rho(y')dy'dy
   \\&\
  -d_{n,p}
  \iint\frac{(w(y')-w(y))^2}{|y'-y|^{n+1}}(y'\cdot\nabla\rho)dy'dy.\end{aligned}
\end{eqnarray}
Integrating \eqref{monohatE1} over $[a,b]$, for any $-\ln T<a<b<\infty$, yields  \eqref{energyinq}. Since $1<p<p_*(n)$, we have $\frac{c_n}{4}<\frac{(p+1)-(p-1)n}{(p+1)M_n}$.
\end{proof}
Since
\begin{eqnarray}
  (-t)^\beta\lambda^\beta u(\lambda x, \lambda t)=w(\lambda y, s-\ln\lambda),
\end{eqnarray}
we will work with weak limits
\begin{eqnarray}
\begin{aligned}
  w_{\infty}(y,s)=&\ \lim_{s_j\rightarrow\infty}w(y, s+s_j),\\
  w_{-\infty}(y,s)=&\ \lim_{s_j\rightarrow-\infty}w(y, s+s_j'),\end{aligned}
\end{eqnarray}
with $s_j=-\ln\lambda_j, s_j'=-\ln\lambda_j'$. Asserting that $u_0$ and $u_{\infty}$
are self-similar is the same as saying that $w_{\infty}$ and $w_{-\infty}$
are independent of $s$, and this is  the main conclusion of the following proposition.

\begin{proposition}\label{constlimitw}
 Let $w$ be a bounded solution of (\ref{halfWE}) on $\mathbb{R}^{n+1}$ satisfying
   (\ref{boundygradw}) and let $\{s_{j}\}$ be a sequence
  such that
  $$s_{j}\rightarrow\pm\infty,\quad s_{j+1}-s_{j}\rightarrow\pm\infty\quad\text{as }j\rightarrow\infty.$$
  Assume that $w_j(y,s)=w(y,s+s_j)$ converges to a limit $w_{\pm\infty}(y,s)$  uniformly on compact subsets of $\mathbb{R}^{n+1}$.
If $1< p<p_*(n)$ and $n\leq4$,
  then the limit $w_{\pm\infty}$
   is independent of $s$, and equal to
  $0, \beta^\beta$ or $-\beta^\beta$.
  Also, the energy $\hat{E}[w_{\pm\infty}]$  is independent
  of the choice of the sequence $\{s_j\}$.
\end{proposition}
\begin{proof}
We shall discuss only the case $s_j\rightarrow+\infty$;
  the proof for $s_j\rightarrow -\infty$ would be derived in the same way.

 Selecting $w=w_j,  b=a+s_{j+1}-s_j$ in Proposition \ref{monotonicity}
  for any real number $a$,
  we obtain
  \begin{eqnarray}
   \begin{aligned}
  \delta_j\hat{E}[w](a):=&\
    \hat{E}[w_j](a)-\hat{E}[w_{j+1}](a)
    \\=&\
    \hat{E}[w_j](a)-\hat{E}[w_j](a+s_{j+1}-s_j)
     \\ \geq&\
    \big(1-\frac{c_n}{4\epsilon}\big)
    \int_a^{b}\int(w_{js}+y\cdot\nabla w_j)^2\rho dyds
    \\&\
   +d_{n,p,\epsilon}
  \int_a^{b}\iint\frac{(w_j(y')-w_j(y))^2}{|y'-y|^{n+1}}\rho(y')dy'dyds
   \\&\
  -d_{n,p}\int_a^{b}
 \iint\frac{(w_j(y')-w_j(y))^2}{|y'-y|^{n+1}}(y'\cdot\nabla\rho)dy'dyds
  \end{aligned}
  \end{eqnarray}
  with $d_{n,p,\epsilon}=(\frac{1-M_n\epsilon}{4}-\frac{(p-1)n}{4(p+1)})c_n$ and
 $d_{n,p}=\frac{(p-1)c_n}{4(p+1)}$. Now we choose $\epsilon=\frac{p+1-(p-1)n}{(p+1)M_n}$,
 then $d_{n,p,\epsilon}=0$ and $1-\frac{c_n}{4\epsilon}>0$ when $1<p<p_*(n)$.

 Thanks to the bounds of $w$ and $\nabla w$, we know that $w_j$ and $\nabla w_j$
 are bounded independently of $j$. Therefore,
  \begin{eqnarray}\label{Econvergence}
    \begin{aligned}
   & \iint\frac{(w_j(y')-w_j(y))^2}{|y'-y|^{n+1}}\rho(y')dy'dy
   \\&\ =\lim_{j\rightarrow\infty}
    \iint\frac{(w_{\infty}(y')-w_{\infty}(y))^2}{|y'-y|^{n+1}}\rho(y')dy'dy.
    \end{aligned}
  \end{eqnarray}
   Arguing similarly for the other terms we see that
  $\hat{E}[w_j](a)\rightarrow \hat{E}[w_\infty](a)$ as $j\rightarrow\infty$. In particular, we know that
  $\delta_j\hat{E}[w](a)\rightarrow0$.
  Since $s_{j+1}-s_j\rightarrow\infty$, it follows that
  \begin{eqnarray}\label{limitwintegral}
    \begin{aligned}
   0\geq&\
    \lim_{j\rightarrow\infty}\{
    \big(1-\frac{c_n}{4\epsilon}\big)
    \int_a^{a'}\int(w_{js}+y\cdot\nabla w_j)^2\rho dyds
  \\&\
  -d_{n,p}\int_a^{a'}
 \iint\frac{(w_j(y')-w_j(y))^2}{|y'-y|^{n+1}}(y'\cdot\nabla\rho)dy'dyds\},
  \end{aligned}
  \end{eqnarray}
  for any real number $a<a'$. (\ref{Econvergence}) and (\ref{limitwintegral}) imply that
  \begin{eqnarray}
    \int_a^{a'}
    \iint\frac{(w_{\infty}(y')-w_{\infty}(y))^2}{|y'-y|^{n+1}}(y'
    \cdot\nabla\rho)dy'dyds
    =0,
  \end{eqnarray}
  for any real number $a<a'$. Thus,
  \begin{eqnarray}\label{limitw}
    w_{\infty}(y,s)\equiv C(s), \ \forall y\in\mathbb{R}^n.
  \end{eqnarray}
  Using the dominated convergence theorem and using integration by parts, we conclude that
  $w_{\infty}=C(s)$ is a weak solution of
  \begin{eqnarray}\label{weakwinf}
    w_{ s}+y\cdot\nabla w+(-\Delta)^{\frac{1}{2}}w+\beta w-|w|^{p-1}w=0.
  \end{eqnarray}
  By (\ref{limitwintegral}), we have
  \begin{eqnarray}\label{limitws}
    \lim_{j\rightarrow\infty}\int_a^{a'}\int(w_{js}+y\cdot\nabla  w_j)^2\rho dyds=
    0
  \end{eqnarray}
  for any real number $a<a'$.
  Now, we have that $|w_{js}+y\cdot\nabla  w_j|\leq C$ with $C$ independent of $j$
  from (\ref{boundw2}),
  and $w_{js}+y\cdot\nabla  w_j$ converges weakly to $w_{\infty s}$ by the equation (\ref{weakwinf}).
  Then the integral (\ref{limitws}) is lower-semicontinuous and we obtain
  \begin{eqnarray}
    \int\int_a^{a'}|w_{\infty s}|^2\rho dsdy=0.
  \end{eqnarray}
  Since $a$ and $a'$ are arbitrary, this means that $w_{\infty}$ is independent of $s$;
  thus we can conclude that $w_{\infty}$ equals to $0, \beta^\beta$ or $-\beta^\beta$
  by means of equation (\ref{weakwinf}).

  It remains to prove that $\hat{E}[w_{\infty}]$ is independent of the choice of sequence.
  If it is not true, then there is another sequence $\{\widetilde{s}_j\}$
  satisfying the hypotheses
  of the proposition for which $\hat{E}[w_\infty]\neq \hat{E}[\widetilde{w}_\infty]$,
  where $\widetilde{w}_\infty=\lim\limits_{j\rightarrow\infty}\widetilde{w}_j$ with
  $\widetilde{w}_j(y,s)=w(y,s+\widetilde{s}_j)$. Relabeling and passing to a subsequence
  if necessary, we may suppose that $\hat{E}[w_\infty]<\hat{E}[\widetilde{w}_\infty]$ and
  $s_j<\widetilde{s}_j$.

  Selecting $a=s_j, b=\widetilde{s}_j$ in Proposition \ref{monotonicity},
  we obtain
  \begin{eqnarray}\label{hatEinq}
   \begin{aligned}
    \hat{E}[w_j](0)- \hat{E}[\tilde{w}_j](0)=&\
    \hat{E}[w](s_j)- \hat{E}[w](\widetilde{s}_j)
    \\\geq&\
    \big(1-\frac{c_n}{4\epsilon}\big)\int_{s_j}^{\widetilde{s}_j}
    \int(w_{s}+y\cdot\nabla w)^2\rho dyds.
    \end{aligned}
  \end{eqnarray}
  Since $\hat{E}[w_j](0)- \hat{E}[\tilde{w}_j](0)\rightarrow
  \hat{E}[w_\infty]-\hat{E}[\widetilde{w}_\infty]<0$, the left hand side of (\ref{hatEinq})
  is negative for sufficiently large $j$. This is a contradiction, because of the right
  side is non-negative. Hence $\hat{E}[w_\infty]= \hat{E}[\widetilde{w}_\infty]$, and
  the proof is complete.
\end{proof}

\begin{proposition}\label{limitsolw}

  Let $w$ be a bounded solution of (\ref{halfWE}) in $\mathbb{R}^{n+1}$
  satisying (\ref{boundygradw}). If
 $1< p<p_*(n)$ and $n\leq4$,
  then $\lim\limits_{s\rightarrow\pm\infty}w(y,s)$ exists and equals to
  $0, \beta^\beta$ or $-\beta^\beta$.
  The convergence is uniform on every compact subset of $\mathbb{R}^n$.
\end{proposition}
\begin{proof}
  Let $\{s_j\}$ be a sequence such that $\lim_{j\rightarrow\infty}s_{j}=\infty$. Since $|\nabla w|$ and $|w_s|$
  are bounded on compact sets by Proposition \ref{propboundw}, there is a subsequence of $\{w(y,s+s_j)\}$
  which converges uniformly on
  compact sets to some function $w_\infty(y,s)$.
  By taking another subsequence if
  necessary, we may assume that $s_{j+1}-s_j\rightarrow\infty$. Therefore, Proposition
  \ref{constlimitw} tells us that $w_\infty$ equals to
  $0$ or $\pm\beta^\beta$.

  Now we prove that the limit is independent of the choice of sequence.
  Suppose that $\{s_j\}$ and $\{\widetilde{s}_j\}$ both tend to infinity and
  satisfy the hypotheses of Proposition \ref{constlimitw}, with
  \begin{eqnarray}
    w_j(y,s)=w(y,s+s_j)\rightarrow w_\infty
    \ \  \text{and}\  \
    \widetilde{w}_j(y,s)=w(y,s+\widetilde{s}_j)\rightarrow\widetilde{w}_\infty.
    \nonumber
  \end{eqnarray}
  It is easy to see that
  \begin{eqnarray}\label{inq1}
    \hat{E}[\pm\beta^\beta]=\frac{\beta^{2\beta}}{2(p+1)}\int\rho dy>0=
    \hat{E}[0].
  \end{eqnarray}
  In order to get \eqref{inq1}, we have used the fact that
  \begin{eqnarray}
   \int(-\Delta)^{\frac{1}{2}}\rho dy=0.
  \end{eqnarray}
  Since $\hat{E}[w_\infty]=\hat{E}[\widetilde{w}_\infty]$ by Proposition \ref{constlimitw}
 , we see that $w_\infty, \widetilde{w}_\infty$ are either both $0$ or both
  $\pm\beta^\beta$. If $w_\infty=\beta^\beta, \widetilde{w}_\infty=-\beta^\beta$
  or vice versa, then there must be a sequence  $s'_j\rightarrow\infty$ with
  $w(0,s'_j)=0$, by the continuity of $w$. Taking subsequences as before (and
  denoting the result again by $s'_j$) we can see that $w'_j=w(y,s+s'_j)
  \rightarrow w'_\infty =0$. This contradicts the fact that $\hat{E}[w_\infty]
  =\hat{E}[w'_\infty]$, and it follows that $w_\infty =\widetilde{w}_\infty$.

   By a parallel way, we can also prove results for the case $s\rightarrow-\infty$.
\end{proof}

\begin{proof}[Proof of Theorem \ref{classifytypeI} ]
 Proposition \ref{limitsolw} implies  Theorem \ref{classifytypeI} as we only consider the
case $s\to\infty$.
\end{proof}
Furthermore, we have

\begin{theorem}\label{thm4.4}
   Let  $w$ be a bounded, global solution of (\ref{halfWE}) in
  $\mathbb{R}^{n+1}$ satisfying (\ref{boundygradw}), with
  \begin{eqnarray}\label{positive}
    \limsup_{s\rightarrow\infty}|w(0,s)|>0.
  \end{eqnarray}
  If $1<p< p_*(n)$ and $n\leq4$, then
  $w\equiv\pm\beta^\beta$.
\end{theorem}
\begin{proof}
  Replacing $w$ by $-w$ if necessary, we may assume that
  \begin{eqnarray}
    \limsup_{s\rightarrow\infty}w(0,s)>0.
  \end{eqnarray}
  Let $\epsilon=\frac{p+1-(p-1)n}{(p+1)M_n}$ in (\ref{energyinq}).
  For any $\tau>0$,
  \begin{eqnarray}
   \begin{aligned}
    \hat{E}[w](-\tau)- \hat{E}[w](\tau)
    \geq&\
    \big(1-\frac{c_n}{4\epsilon}\big)\int_{-\tau}^{\tau}
    \int(w_{s}+y\cdot\nabla  w)^2\rho dyds
     \\&\
    -d_{n,p}\int_{-\tau}^{\tau}\iint\frac{(w(y')-w(y))^2}{|y'-y|^{n+1}}(y'\cdot\nabla\rho)dy'dyds.
  \end{aligned}
  \end{eqnarray}
  Applying Proposition \ref{limitsolw} and passing to the limit,
  \begin{eqnarray}\label{inq}
    \begin{aligned}
    \hat{E}[w_{-\infty}]-\hat{E}[w_\infty]
    \geq&\
    \big(1-\frac{c_n}{4\epsilon}\big)\int_{-\infty}^{\infty}
    \int(w_{s}+y\cdot\nabla  w)^2\rho dyds
     \\&\
  -d_{n,p}\int_{-\infty}^{\infty}\iint\frac{(w(y')-w(y))^2}{|y'-y|^{n+1}}(y'\cdot\nabla\rho)dy'dyds,
  \end{aligned}
  \end{eqnarray}
  with $w_{\infty}=\lim\limits_{s\rightarrow\infty}w(y,s)$ (respectively $w_{-\infty}$), and the hypothesis on $w$ assures that $w_\infty=\beta^\beta$. If $w_{-\infty}=0$, then the right side
  of (\ref{inq}) would be negative by (\ref{inq1}), which cannot happen. Therefore, $w_{-\infty}=\pm\beta^\beta$, and then (\ref{inq}) implies that $w\equiv C$.
   It follows that
  $w=w_\infty=\beta^\beta$, and we complete the proof.
\end{proof}
\begin{proof}[Proof of Theorem \ref{halfliouville}]
Theorem \ref{halfliouville} is just an equivalent statement of Theorem \ref{thm4.4}  for $u$.
\end{proof}

\section{Proof of  Theorem \ref{classifyall}: Growth rate estimate for $1<p<p_*(n)$}
Let $T<\infty$ and let $u:\mathbb{R}^n\times(-T,0)\rightarrow \mathbb{R}$ be
a  classical solution of the semi-linear half heat equation
\begin{eqnarray}\label{halfHE1}
  u_t+(-\Delta)^{\frac{1}{2}}u-|u|^{p-1}u=0
\end{eqnarray}
such that, for any $-T<t'<0$,
\begin{eqnarray}\label{bounduwitht}
   \begin{aligned}
   &u,\nabla u, \nabla^2 u, \nabla^3u, \ \text{and}\ u_t\
   \text{are bounded and continuous}\\
   &\text{on}\ \mathbb{R}^n\times[-T,t'].
  \end{aligned}
\end{eqnarray}
We recall that a solution $u$ blows up at $t=0$ if
\begin{eqnarray}\label{ublowat0}
  \sup_{x\in\mathbb{R}^n}|u(x,t)|\rightarrow\infty, \quad
  \text{as}\ t\rightarrow 0.
\end{eqnarray}
In this section, we want to obtain the following blow up rate estimate:
\begin{eqnarray}\label{blowrate1}
  \sup_{\R^n\times(-T,0)}\{|u|^{p-1}+|\nabla u|^{\frac{p-1}{p}}+|\nabla^{2}u|^{\frac{p-1}{2p-1}}\}(-t)<\infty.
\end{eqnarray}
As before, we also need a further assumption: fix $\delta>0$, for any $-T<t'<0$,
\begin{eqnarray}\label{boundxgraduwitht}
  |\nabla u|(1+|x|^{\delta}) \ \text{is bounded on}\
  \mathbb{R}^n\times[-T,t'].
\end{eqnarray}

Assume that estimate \eqref{blowrate1} fails. Since $T<\infty$, there is an increasing sequence of times $t_{k}\rightarrow 0$ such that
\begin{eqnarray}\label{tkseq}
  \begin{aligned}
 &\sup\limits_{\mathbb{R}^n\times (-T, t_{k})}\{|u|^{p-1}+|\nabla u|^{\frac{p-1}{p}}+|\nabla^{2}u|^{\frac{p-1}{2p-1}}\}(-t)\\
 &= \sup\limits_{\mathbb{R}^n}\{|u|^{p-1}(x, t_{k})+|\nabla u(x, t_{k})|^{\frac{p-1}{p}}+|\nabla^{2}u(x, t_{k})|^{\frac{p-1}{2p-1}}\}(-t_{k})=M_{k}
 \end{aligned}
\end{eqnarray}
with $t_k\rightarrow 0$ and $M_k\rightarrow \infty$ as $k\rightarrow \infty$. We may choose a sequence $x_k\in\mathbb{R}^n$ such that
\begin{eqnarray}\label{xkseq}
 \frac{1}{2}M_{k}\leq \{|u|^{p-1}(x_{k}, t_{k})+|\nabla u(x_{k}, t_{k})|^{\frac{p-1}{p}}+|\nabla^{2}u(x_{k}, t_{k})|^{\frac{p-1}{2p-1}}\}(-t_k)\leq M_{k}.
\end{eqnarray}
In order to study $u$ near $(x_k,t_k)$, we use the \emph{similarity variables} defined by
\begin{eqnarray}\label{wkseq}
  w_k(y,s)=(-t)^{\beta}u(x,t),
\end{eqnarray}
\begin{eqnarray}\label{simvariable}
  x-x_k=(-t)y,\ t=-e^{-s}.
\end{eqnarray}
It is easy to verify that the rescaled function $w_k$ satisfies
 \begin{eqnarray}\label{halfWkE}
   w_{ks}+(-\Delta)^{\frac{1}{2}}w_k+y\cdot\nabla w_k+\beta w_k-|w_k|^{p-1}w_k
   =0,
 \end{eqnarray}
 and  inherits bounds from those on $u$: for any $-\ln T<s'<\infty$,
\begin{eqnarray}\label{boundwkwiths}
   \begin{aligned}
   &w_k,\nabla w_k, \nabla^2 w_k, \nabla^3w_k,
   w_{ks}+y\cdot\nabla  w_k\ \text{and}\ \nabla(w_{ks}+y\cdot\nabla  w_k)\
   \text{are }\\
   &\text{bounded and continuous on}\ \mathbb{R}^n\times[-\ln T,s'].
  \end{aligned}
\end{eqnarray}
 By (\ref{boundxgraduwitht}), we know that $w_{k}$ also satisfies: for any $-\ln T<s'<\infty$,
\begin{eqnarray}\label{boundygradwkwiths}
   \begin{aligned}
    |\nabla  w_k(y,s)|(1+|y|^{\delta})<\infty,\ \
  \text{on}\ \mathbb{R}^n\times[-\ln T,s'].
  \end{aligned}
\end{eqnarray}
\begin{remark}
In the following computations, the integration by parts can be justified if we assume (\ref{boundwkwiths}) and (\ref{boundygradwkwiths}).
\end{remark}
\begin{proposition}\label{propEhatE}
  For any $k$, the rescaled solution $w_k$ satisfies
  \begin{eqnarray}\label{energyidwk}
   \begin{aligned}
    \frac{1}{2}\frac{d}{ds}
    \int |w_k|^2\rho dy=&\ \frac{1}{p+1}\int|w_k|^2(-\Delta)^{\frac{1}
    {2}}\rho dy
    \\&\
    -2\hat{E}[w_k](s)+\frac{p-1}{p+1}\int|w_k|^{p+1}
    \rho dy
    \end{aligned}
  \end{eqnarray}
  and
    \begin{eqnarray}\label{monohatEk}
     \begin{aligned}
    -\frac{d\hat{E}[w_k]}{ds}\geq&\
      \big(1-\frac{c_n}{4\epsilon}\big)\int(w_{ks}+y\cdot\nabla  w_k)^2\rho dy
     \\&\
+d_{n,p,\epsilon}
  \iint\frac{(w_k(y')-w_k(y))^2}{|y'-y|^{n+1}}\rho(y')dy'dy
   \\&\
  -d_{n,p}\iint\frac{(w_k(y')-w_k(y))^2}{|y'-y|^{n+1}}(y'\cdot\nabla\rho)dy'dy,
    \end{aligned}
  \end{eqnarray}
where $\epsilon\in(\frac{c_n}{4}, \frac{(p+1)-(p-1)n}{(p+1)M_n})$, $d_{n,p,\epsilon}=(\frac{1-M_n\epsilon}{4}-\frac{(p-1)n}{4(p+1)})c_n>0$ and
  $d_{n,p}=\frac{(p-1)c_n}{4(p+1)}$. $\hat{E}[w](s)$
  is defined by   (\ref{quantityhatE}).
\end{proposition}
\begin{proof}
  Since
  \begin{eqnarray}
    \frac{d}{ds}\int|w_k|^2\rho dy=2\int w_kw_{ks}\rho dy,
  \end{eqnarray}
  multiplying the equation (\ref{halfWkE}) by
  $\rho w_k$ and using integration by parts, by the definition of $\hat{E}[w](s)$,
  it is easy  to get (\ref{energyidwk}).

 Inequality (\ref{monohatEk}) is merely a restatement of (\ref{monohatE1}) with $w$ replaced by $w_k$.
\end{proof}

\begin{proposition}
  If   $1<p<p_*(n)$ and  $n\leq4$,
  there exists a constant $C$  such that
  \begin{eqnarray}\label{Ekbound}
  \hat{E}[w_k](-\ln T)\leq C,\ \text{for all}\ k.
  \end{eqnarray}
  We also have
  \begin{eqnarray}\label{boundenergy}
   \begin{aligned}
   M'\geq&\
    \hat{E}[w_k](-\ln T)-\hat{E}[w_k](\infty)
    \\\geq&\
       \big(1-\frac{c_n}{4\epsilon}\big)\int_{-\ln T}^{\infty}
    \int(w_{ks}+y\cdot\nabla  w_k)^2\rho dyds
        \\&\
+d_{n,p,\epsilon}
  \int_{-\ln T}^{\infty}\iint\frac{(w_k(y')-w_k(y))^2}{|y'-y|^{n+1}}\rho(y')dy'dyds
   \\&\
    -d_{n,p}\int_{-\ln T}^{\infty}\iint\frac{(w_k(y')-w_k(y))^2}{|y'-y|^{n+1}}(y'\cdot\nabla\rho)dy'dyds,
    \end{aligned}
  \end{eqnarray}
where $\epsilon\in(\frac{c_n}{4}, \frac{(p+1)-(p-1)n}{(p+1)M_n})$, $d_{n,p,\epsilon}=(\frac{1-M_n\epsilon}{4}-\frac{(p-1)n}{4(p+1)})c_n>0$ and
  $d_{n,p}=\frac{(p-1)c_n}{4(p+1)}$. In \eqref{boundenergy}, the constant $M'$ depends only on
  $n, p$ and the bound $C$ in \eqref{Ekbound} which is independent of $k$.
\end{proposition}
\begin{proof}
We have assumed that $T<\infty$ and $u(x,-T)$ and its derivatives up to  order three are bounded on $\R^n$; see \eqref{bounduwitht}. Recall that
\begin{eqnarray}
w_k(y,-\ln T)=T^\beta u(x_k+Ty,-T)=:T^\beta U(y),\nonumber
\end{eqnarray}
then we have, for $s=-\ln T$,
\begin{eqnarray*}
\begin{aligned}
\iint\frac{(w_k(y')-w_k(y))^2}{|y'-y|^{n+1}}\rho(y')dy'dy=&\ T^{2\beta} \iint\frac{(U(y')-U(y))^2}{|y'-y|^{n+1}}\rho(y')dy'dy\\
\leq&\
T^{2\beta+2} \sup_{\R^n\times\{-T\}}|\nabla u|^2\iint_{B_1(y')}\frac{\rho(y')}{|y'-y|^{n-1}}dydy'
    \nonumber\\&\
    +4T^{2\beta} \sup_{\R^n\times\{-T\}}|u|^2\iint_{\mathbb{R}^n\setminus B_1(y')}\frac{\rho(y')}{|y'-y|^{n+1}}dydy'\\
    \leq&\ C<\infty.
\end{aligned}
\end{eqnarray*}
The other terms in $\hat{E}[w_k](-\ln T)$ are handled similarly. It is easy to obtain \eqref{Ekbound}.

 Inequality (\ref{monohatEk}) in Proposition \ref{propEhatE} implies that
  $\hat{E}[w_k](s)$ is a decreasing funcition of $s$ when
  $p<p_*(n)$.

  We claim that there exists a positive constant $A$ such that $\hat{E}[w_k](s)\geq-A$  for all $s$.

  Recall that
  \begin{eqnarray}\label{lowboundhalfrho}
   \begin{aligned}
    (-\Delta)^{\frac{1}{2}}\rho =&\
    n\rho+y\cdot\nabla\rho
    \\ =&\
    \big(n-(n+1)\frac{|y|^2}{1+|y|^2}\big)\rho
    \\\geq&\
    -\rho.
    \end{aligned}
  \end{eqnarray}
   Let
  \begin{eqnarray}\label{qn}
    q_n:=\int\rho dy=\frac{\pi^{\frac{n+1}{2}}}{\Gamma(\frac{n+1}{2})},
    \end{eqnarray}
    and let
  \begin{eqnarray}
    g(s)=\Big(\int |w_k|^2\rho dy\Big)^{\frac{1}{2}}.
  \end{eqnarray}
    By (\ref{energyidwk}), (\ref{lowboundhalfrho}) and Jensen's inequality, we know that
  \begin{eqnarray}\label{energyidwk1}
    \frac{1}{2}\frac{d}{ds}\big(g^2(s)\big)
     &\geq&-\frac{1}{p+1}g^2(s)
    -2\hat{E}[w_k](s)+c_{n,p}g^{p+1}(s),
    \nonumber
  \end{eqnarray}
 where $$c_{n,p}=\frac{p-1}{p+1}q_n^{-\frac{n(p-1)}{4}}.$$
 Now we choose
  \begin{eqnarray}
    A=\frac{1}{p+1}\Big(\frac{2}{(p+1)c_{n,p}}\Big)^{\frac{2}{p-1}}.
  \end{eqnarray}
  Suppose there exists a constant $s_{1}$ such that $\hat{E}[w_k](s_1)<-A$, then
  \begin{eqnarray}\label{qqqq}
  \hat{E}[w_k](s)<-A \quad \text{for all}\ s>s_1.
  \end{eqnarray}
   In \eqref{qqqq},  we have applied (\ref{monohatEk}).
If $g(s)<\Big(\frac{2}{(p+1)c_{n,p}}\Big)^{\frac{1}{p-1}}$, then
  \begin{eqnarray}
    \frac{1}{2}\frac{d}{ds}\big(g^2(s)\big)\geq A+c_{n,p}g^{p+1}(s).
  \end{eqnarray}
  If $g(s)\geq\Big(\frac{2}{(p+1)c_{n,p}}\Big)^{\frac{2}{p-1}}$, then
    \begin{eqnarray}
    \frac{1}{2}\frac{d}{ds}\big(g^2(s)\big)\geq 2A+\frac{1}{2}c_{n,p}g^{p+1}(s).
  \end{eqnarray}
  Therefore, we conclude that
  \begin{eqnarray}\label{qw}
    \frac{1}{2}\frac{d}{ds}\big(g^2(s)\big)\geq A+\frac{1}{2}c_{n,p}g^{p+1}(s)\quad\text{for all }s>s_{1}.
  \end{eqnarray}
  \eqref{qw} implies that $g(s)$ blows up in finite time, contradicting the global existence
  of $w_k$. Therefore,
  \begin{eqnarray}\label{qw'}
   - A\leq \hat{E}[w_k](s)\leq\hat{E}[w_k](-\ln T)\quad\text{for all }s<\infty.
  \end{eqnarray}
   \eqref{qw'} and Proposition \ref{monotonicity} imply that $\hat{E}[w_k](s)$  have a limit $\hat{E}[w_k](\infty)$$\geq- A$ as $s\rightarrow\infty$. Integrating (\ref{monohatEk}) gives (\ref{boundenergy}) since we may set $M'=C+A$.
\end{proof}
\begin{proof}[Proof of Theorem \ref{classifyall}]
  If (\ref{blowrate1}) fails, then (\ref{tkseq}) and (\ref{xkseq}) are true. Let $w_k(y,s)$
  be the recaled solution around $x_k$ defined by (\ref{wkseq}). Setting $s_k=-\ln(-t_k)$,
  (\ref{tkseq}) and (\ref{xkseq}) become
  \begin{eqnarray}\label{skseq}
    |w_{k}|^{p-1}(y, s)+|\nabla w_{k}(y, s)|^{\frac{p-1}{p}}+|\nabla^{2}
    w_{k}(y, s)|^{\frac{p-1}{2p-1}}\leq M_{k},\ -\ln T<s\leq s_{k}
  \end{eqnarray}
  and
  \begin{equation}\label{ykseq}
 \frac{1}{2}M_{k}\leq w_{k}^{p-1}(0, s_{k})+|\nabla w_{k}(0, s_{k})|^{\frac{p-1}{p}}+|\nabla^{2} w_{k}(0, s_{k})|^{\frac{p-1}{2p-1}}\leq M_{k}.
 \end{equation}
  We may therefore rescale $\{w_k\}$, defining
  \begin{eqnarray}
    v_k(z,\tau)=\lambda_k^\beta w_k(\lambda_kz,\lambda_k\tau+s_k),\quad
    \beta=\frac{1}{p-1},
  \end{eqnarray}
  with $\lambda_k\rightarrow0$ determined by $\lambda_k^\beta M_k=1$.  Each $v_{k}$ is defined on $\mathbb{R}^n\times(-\frac{1}{\lambda_k},0]$,
 and (\ref{skseq}) and (\ref{ykseq})
   gives
  \begin{eqnarray}\label{taukseq}
     |v_{k}|^{p-1}+|\nabla v_{k}|^{\frac{p-1}{p}}+|\nabla^{2} v_{k}|^{\frac{p-1}{2p-1}}\leq 1\quad\text{on}\quad
     \mathbb{R}^n
     \times(-\frac{1}{\lambda_k},0]
  \end{eqnarray}
  and
  \begin{eqnarray}\label{zkseq}
     \frac{1}{2}\leq |v_{k}|^{p-1}(0, 0)+|\nabla v_{k}(0, 0)|^{\frac{p-1}{p}}+|\nabla^{2} v_{k}(0, 0)|^{\frac{p-1}{2p-1}}\leq 1.
  \end{eqnarray}
  The equation for $v_k$ is obtained by changing variables in (\ref{halfWkE})
   for $w_k$:
  \begin{eqnarray}\label{halfVkE}
    v_{k\tau}+(-\Delta)^{\frac{1}{2}}v_k-|v_k|^{p-1}v_k=-\lambda_k(z\cdot\nabla
    v_k+\beta v_k).
  \end{eqnarray}
By \eqref{zkseq} and the equation \eqref{halfVkE}, we know that $v_{k\tau}$ are uniformly bounded on $\mathbb{R}^n\times(-\frac{1}{\lambda_k},0]$.

We denote
\begin{equation}\label{Qr}
 Q(r)=\{(z, \tau)\in\mathbb{R}^{n+1}:|z|<r,-r\leq\tau\leq 0\}
 \end{equation}
 and
 \begin{equation}\label{Qrprime}
 Q'(r)=\{(z, \tau)\in\mathbb{R}^{n+1}:|z|<r,-r\leq\tau< 0\}.
 \end{equation}
For each $r$, there exists a large $K$, such that $Q(r)\subset\subset\R^n\times(-\frac{1}{\lambda_k}, 0]$ for all $k\geq K$. To get the limit equation of $v_k$, we need some regularities independent of $k$
in each $Q(r)$.
By Theorem 1.3 in \cite{FR} and \eqref{taukseq}, we can get
\begin{eqnarray}\label{vkholder1}
  |v_k|_{C_{z,\tau}^{\frac{1}{2},\frac{1}{2}}(Q'(r))}\leq C
\end{eqnarray}
where constant $C$ independent of $k$, and following the notations in \cite{FR} as
 \begin{eqnarray}
       [v_k]_{C_{z,\tau}^{\frac{1}{2},\frac{1}{2}}(\Omega\times I)}:=
       \sup_{\begin{subarray}{1}
       (z,\tau)\in\Omega\times I\\
       (z',\tau')\in\Omega\times I\end{subarray}}
       \frac{|v_k(z,\tau)-v_k(z',\tau')|}{|z-z'|^{\frac{1}{2}}
       +|\tau-\tau'|^{\frac{1}{2}}}
     \end{eqnarray}
for any $\Omega\times I\subset\mathbb{R}^{n+1}$.
By continuity and \eqref{taukseq}, we can obtain
\begin{equation}\label{vkholder}
|v_k|_{C_{z,\tau}^{\frac{1}{2},\frac{1}{2}}(Q(r))}\leq C.
\end{equation}
Taking the derivative with respect to $z$ in \eqref{halfVkE} yields
\begin{equation}\label{halfgradVkE}
\partial_{z}v_{k\tau}+(-\Delta)^{\frac{1}{2}}\partial_{z}v_{k}=pv_{k}^{p-1}
\partial_{z}v_{k}-\lambda_{k}\partial_{z}(z\cdot\nabla v_{k}+\beta v_{k}).
\end{equation}
Similar to the discussion used to derive \eqref{vkholder}, we could  obtain
\begin{equation}\label{vkholder2}
|\nabla v_k|_{C_{z,\tau}^{\frac{1}{2},\frac{1}{2}}(Q(r))}\leq C.
\end{equation}
Moreover, by an iteration argument, we could also obtain
\begin{equation}\label{vkholder3}
|\nabla^2v_k|_{C_{z,\tau}^{\frac{1}{2},\frac{1}{2}}(Q(r))}\leq C.
\end{equation}
At last, we need to prove that
\begin{equation}\label{vktauholder}
|v_{k\tau}|_{C_{z,\tau}^{\frac{1}{2},\frac{1}{2}}(Q(r))}\leq C
\end{equation}
with $C$ independent of $k$.
Let $\psi$ be a smooth function such that
$$
 0\leq\psi\leq 1,\qquad\text{supp}\psi\subset Q(2r),
 \qquad\psi\equiv1\quad\text{on}\quad Q(r).
$$
and let $\tilde{v}_k(z, \tau)$ be the function defined by $\tilde{v}_k(z, \tau)=\psi(z, \tau) v_{k}(z, \tau)$. It is easy to verify that $\tilde{v}_k$ satisfies the equation
\begin{equation}\label{23}
 \tilde{v}_{k\tau}+(-\Delta)^{\frac{1}{2}}\tilde{v}_k=f_k,
\end{equation}
here the function $f_k$ is given by
\begin{eqnarray}
 \begin{aligned}
f_k=&\ [(-\Delta)^{\frac{1}{2}}v_{k\tau}+v_{k\tau}]\psi+[(-\Delta)^{\frac{1}{2}}\psi
+\psi_{\tau}]v_{k}
\\&\
+c_n
\int\frac{(v_{k}(z)-v_{k}(y))(\psi(z)-\psi(y))}{|z-y|^{n+1}}dy.
 \end{aligned}
\end{eqnarray}
By \eqref{taukseq} and the definition of $\tilde{v}_k$, we can get that $f_k\in L^{\infty}(Q(2r)),
\tilde{v}_k\in L^{\infty}(\mathbb{R}^n\times(-2r, 0))$.
It follows from Theorem 1.3 in \cite{FR} that $\tilde{v}_k\in C_{\tau}^{\frac{1}{2}}(Q'(2r))$.
Since $\tilde{v}_k=0$ on
$\mathbb{R}^n\times(-2r, 0)\backslash Q(2r)$.
It is easy to see that $\tilde{v}\in C^{\frac{1}{2}, \frac{1}{2}}_{z, \tau}(\mathbb{R}^n\times(-2r, 0))$.
We deduce from Theorem 1.1 in \cite{FR}
 that $\tilde{v}_{\tau}\in C^{\frac{1}{2}, \frac{1}{2}}_{z, \tau}(Q'(2r))$.
 In particular, we have
 \begin{equation}
|v_{k\tau}|_{C_{z,\tau}^{\frac{1}{2},\frac{1}{2}}(Q'(r))}\leq C
\end{equation}
with $C$ independent of $k$.
By the continuity, the equation \eqref{halfgradVkE} and \eqref{vkholder3}, we could also get
\eqref{vktauholder}.

  Thanks to \eqref{vkholder}, \eqref{vkholder2}, \eqref{vkholder3} and \eqref{vktauholder}, we can
   use the Arzela-Ascoli theorem and a diagonal argument to get a subsequence (still
  denoted $v_k$) converging uniformly to a limit $v$ on each $Q(r)$.
  This $v$ is defined on $\mathbb{R}^n\times(-\infty,0]$, and it satisfies
  \begin{eqnarray}\label{halfVE}
  \left\{\begin{aligned}
    v_{\tau}+(-\Delta)^{\frac{1}{2}}v-|v|^{p-1}v=0,\\
     \frac{1}{2}\leq \{|v|^{p-1}+|\nabla v|^{\frac{p-1}{p}}+|\nabla^{2} v|^{\frac{p-1}{2p-1}}\}(0,0)\leq 1
    \end{aligned}
    \right.
  \end{eqnarray}
  by passing to the limit in (\ref{taukseq}), (\ref{zkseq}) and (\ref{halfVkE}).

We may assume $\lambda_k\leq1$, then $ \rho(z)\leq \rho(y)$ for $y=\lambda_kz$. In \eqref{boundenergy}, we fix $\epsilon\in(\frac{c_n}{4},\frac{p+1-(p-1)n}{(p+1)M_n})$ such that $d_{n,p,\epsilon}>0$ and $1-\frac{c_n}{4\epsilon}>0$.

\vspace{0.3cm}
  (i)$1<p<p_*(n)$ and $u\geq0$.   By changing variables
  and applying (\ref{boundenergy}) we have
  \begin{eqnarray}\begin{aligned}
    \int_{Q(\frac{1}{\lambda_k})}
    (v_{k\tau}+\lambda_k z\cdot\nabla v_k)^2\rho dzd\tau\leq&\
    \lambda_k^{\sigma'}\int_{\Omega_k}(w_{ks}+y\cdot\nabla w_k)^2\rho dyds
    \\ \leq&\
    \big(1-\frac{c_n}{4\epsilon}\big)^{-1}\lambda_k^{\sigma'}M'
    \end{aligned}
  \end{eqnarray}
  with $\sigma'=2\beta+1-n>0$ since $1<p<p_*(n)<\frac{n+1}{n-1}$, and
  \begin{eqnarray}
  \Omega_k=\{(y,s)\in\R^{n+1}: |y|\leq1, s_k-1\leq s\leq s_k\}.\nonumber
  \end{eqnarray}
  It follows that
  \begin{eqnarray}
    \int_D|v_{\tau}|^2\rho dzd\tau=0
  \end{eqnarray}
  for any compact subset $D\subset\mathbb{R}^n\times(-\infty,0]$. It  means that  $v$
  is independent of $\tau$. By the assumption and maximum principle, $v$ is a positive bounded solution of
  the limit equation,
  \begin{eqnarray}\label{limitv}
    (-\Delta)^{\frac{1}{2}}v=v^p,\quad \text{on}\ \mathbb{R}^n.
  \end{eqnarray}
  The corollary 1 in \cite{CLL} tell us that \eqref{limitv} has no positive bounded
  solution. This is a  contradiction. Thus, (\ref{blowrate1}) is true.

  \vspace{0.3cm}

  (ii)$1<p<\min\{1+\frac{2}{n}, p_*(n)\}$.  By changing variables
  and applying (\ref{boundenergy}) we have
  \begin{eqnarray}
   \begin{aligned}
    &\int_{\tilde{Q}_k}\frac{(v_k(z')-v_k(z))^2}{|z'-z|^{n+1}}\rho(z) dz'dzd\tau
    \\&\ \leq
    \lambda_k^\sigma\int_{\tilde{\Omega}_k}\frac{(w_k(y')-w_k(y))^2}{|y'-y|^{n+1}}\rho(y)
     dy'dyds
    \\&\ \leq
    \lambda_k^\sigma \frac{M'}{d_{n,p,\epsilon}}
    \end{aligned}
  \end{eqnarray}
  with $\sigma=2\beta-n>0$ since $1<p<1+\frac{2}{n}$, and
  \begin{eqnarray}
  \begin{aligned}
  \tilde{Q}_k=&\ \{(z,z',\tau)\in\R^{2n+1}: |z|\leq\frac{1}{\lambda_k}, |z'|\leq \frac{1}{\lambda_k}, -\frac{1}{\lambda_k}<\tau\leq0\},\\
  \tilde{\Omega}_k=&\ \{(y,y',s)\in\R^{2n+1}: |y|\leq1, |y'|\leq 1, s_k-1<s\leq s_k\}.
  \end{aligned}
  \nonumber
  \end{eqnarray}
   Therefore,
  \begin{eqnarray}
    \int_D\frac{(v(z')-v(z))^2}{|z'-z|^{n+1}}\rho(z) dz'dzd\tau
    =0,
  \end{eqnarray}
   for any compact subset $D\subset\mathbb{R}^{2n}\times(-\infty,0]$, which means that $v\equiv C(\tau)$. On the other hand, (\ref{boundenergy})
  also tells us that
  \begin{eqnarray}
   \begin{aligned}
    \int_{Q(\frac{1}{\lambda_k})}(v_{k\tau}+\lambda_k z\cdot\nabla v_k)^2\rho dzd\tau\leq&\
    \lambda_k^{\sigma'}\int_{\Omega_k}(w_{ks}+y\cdot\nabla w_k)^2\rho dyds
    \\\leq&\
    \big(1-\frac{c_n}{4\epsilon}\big)^{-1}\lambda_k^{\sigma'}M'
    \end{aligned}
  \end{eqnarray}
  with $\sigma'=2\beta+1-n>0$.
  It follows that
  \begin{eqnarray}
    \int_D|v_{\tau}|^2\rho dzd\tau=0,
  \end{eqnarray}
  for any compact subset $D\subset\mathbb{R}^n\times(-\infty,0]$, which means that  $v\equiv C$ for some constant $C$. It is easy to see that $C=0$ by the limiting equation
  (\ref{halfVE}), contradicting  $\{|v|^{p-1}+|\nabla v|^{\frac{p-1}{p}}+|\nabla^{2} v|^{\frac{p-1}{2p-1}}\}(0,0)\geq\frac{1}{2}$. Thus, \eqref{blowrate1} is true.

  Therefore, we conclude that \eqref{blowrate1} is true in both cases.
   Applying Theorem
  \ref{classifytypeI} and \eqref{ublowat0}, we complete the proof of Theorem \ref{classifyall}.
\end{proof}

\section{Proof of theorem \ref{classifyCP}}
In order to prove  Theorem \ref{classifyCP}, we need the following lemma.
\begin{lemma}\label{halfeta}
Let $0<\delta<1$ and  let $\eta(x)=(1+|x|)^{-\delta}$,  then there exists a positive constant $c=c(n,\delta)$ such that
\begin{eqnarray}\label{etaestimate}
|(-\Delta)^{\frac{1}{2}}\eta|(x)
\leq c\eta(x) \ \text{for all }x\in\R^n.
\end{eqnarray}
\end{lemma}
\begin{proof}
Let $\eta(x)=(1+|x|)^{-\delta}$, we estimate
$$|P.V.\int_{\mathbb{R}^{n}}\frac{\eta(x)-\eta(y)}{|x-y|^{n+1}}dy|.$$

If $|x|$ is bounded, the inequality \eqref{etaestimate} holds by choosing a suitable constant $c$.

If $|x|$ is large enough, similar to \cite{DDGW}, we decompose the above integral as follows:
$$P.V.\int_{\mathbb{R}^{n}}\frac{\eta(x)-\eta(y)}{|x-y|^{n+1}}dy=I_{1}+I_{2}+I_{3}+I_{4},$$
here
$$
\begin{aligned}
I_{1}&=P.V.\int_{\frac{|x|}{2}\leq|x-y|\leq 2|x|}\frac{\eta(x)-\eta(y)}{|x-y|^{n+1}}dy,\\
I_{2}&=P.V.\int_{1\leq|y-x|\leq\frac{|x|}{2}}\frac{\eta(x)-\eta(y)}{|x-y|^{n+1}}dy,\\
I_{3}&=P.V.\int_{|y-x|\leq 1}\frac{\eta(x)-\eta(y)}{|x-y|^{n+1}}dy,\\
I_{4}&=P.V.\int_{|y-x|>2|x|}\frac{\eta(x)-\eta(y)}{|x-y|^{n+1}}dy.
\end{aligned}$$
For the first integral, we have
$$
\begin{aligned}
|I_{1}|
\leq c|x|^{-(n+1)}\int_{\frac{|x|}{2}\leq|x-y|\leq 2|x|}|\eta(x)-\eta(y)|dy\leq c|x|^{-1}.
\end{aligned}$$
Since
$$\eta(x)-\eta(y)=\nabla\eta(\xi)\cdot(x-y)$$
where $\xi=x+\lambda(y-x)$ with some $\lambda\in[0,1]$. If $1\leq|y-x|\leq\frac{|x|}{2}$, then
$$|\eta(x)-\eta(y)|\leq c(1+|x|)^{-1-\delta}|x-y|.$$
Therefore, the second one satisfies
$$
\begin{aligned}
|I_{2}|
\leq c|x|^{-1-\delta}\int_{1\leq|y-x|\leq\frac{|x|}{2}}|x-y|^{-n}dy
\leq c|x|^{-\delta}.
\end{aligned}
$$
To bound the third one, notice that
$$
\begin{aligned}
|I_3|&= |\int_{|y-x|\leq 1}\frac{\eta(y)-\eta(x)-\nabla \eta(x)\cdot(x-y)}{|x-y|^{n+1}}dy|\\
&\leq c|D^{2}\eta|(x)\int_{|y-x|\leq 1}|x-y|^{1-n}dy\leq c|x|^{-\delta}.
\end{aligned}$$
The last one can be bound as follows.
$$
\begin{aligned}
|I_{4}|
\leq c|x|^{-\delta}\int_{|y-x|>2|x|}|x-y|^{-n-1}dy\leq c|x|^{-\delta}.
\end{aligned}$$
Since $0<\delta<1$, we can get that
$$|P.V.\int_{\mathbb{R}^{n}}\frac{\eta(x)-\eta(y)}{|x-y|^{n+1}}dy|\leq c|x|^{-\delta}\leq2c(1+|x|)^{-\delta}$$
if $|x|$ is large enough. In conclusion, the inequality \eqref{etaestimate} holds for all $x\in\R^n$.
\end{proof}

\begin{proof}[Proof of Theorem \ref{classifyCP}]
Let $h_j=\frac{\partial u}{\partial x_j}, j=1,\cdots,n$, the derivatives of the solution $u$ of the Cauchy problem \eqref{CP}. Then each $h_j$ satisfies the Cauchy problem:
\begin{eqnarray}
\left\{\begin{array}{l}
L[h_j]:=\partial_t h_j+(-\Delta)^{\frac{1}{2}} h_j - p|u|^{p-1} h_j=0, \ (x,t) \in \R^n\times(0,T)\\
h_j(x, 0)= \frac{\partial u_0}{\partial x_j} (x), \ x\in \R^n.
\end{array}
\right.
\end{eqnarray}
By the decay assumption \eqref{u0decay} on the gradient of the initial value $u_0$,  we can take a suitable $\delta_0\in(0,1)$, $\eta(x)=(1+|x|)^{-\delta_0}$ such that
\begin{eqnarray}
\label{hidecay}
|h_j(x,0)|\leq C\eta(x).
\end{eqnarray}
Let $g(x,t)=Ke^{\lambda t}\eta(x)$, here $K$ and $\lambda$ will be determined later. For any $0<T'<T$, by the definition of the finite blow up time $T$, we have
\begin{eqnarray}
M_{T'}:=\sup_{\R^n\times(0,T')}|u(x,t)|<\infty.
\end{eqnarray}
Hence, by Lemma \ref{halfeta}, we get
\begin{eqnarray}
\begin{aligned}
L[g]=&\ Ke^{\lambda t}\eta(\lambda-c-p|u|^{p-1})\\
\geq& \ 0,
\end{aligned}
\end{eqnarray}
provided $\lambda=pM^{p-1}_{T'}+c+1$.  By the decay condition \eqref{hidecay},
\begin{eqnarray}
g(x,0)\geq h_j(x,0),\ x\in\R^n,
\end{eqnarray}
if we choose $K=C+1$.
Then the maximum principle (see lemma 4.1 in \cite{JX2}) tells us that
\begin{equation}
h_j(x,t)\leq g(x,t),\ (x,t)\in\R^n\times (0,T'],
\end{equation}
for any $0<T'<T$. Obviously, $-g$ is a sub barrier for $h_j, j=1,\cdots,n$. In conclusion, we have the decay condition
\begin{eqnarray}
\label{CPdecay}
|\nabla u|(x,t)\leq \frac{C(T')}{1+|x|^{\delta_0}},\ (x,t)\in\R^n\times(0,T'],
\end{eqnarray}
for any $0<T'<T$. $C(T')=Ke^{\lambda T'}$ converges to infinite as $T'\to T$.

On the other hand, since $u_0$ is nonnegative, the maximum principle also tells us that $u\geq0$.
Since the decay condition \eqref{CPdecay} and the nonnegativity of $u$ hold, we can use Theorem  \ref{classifyall} to prove Theorem \ref{classifyCP}.
\end{proof}
Finally, we want to show that for some special  initial value problems, the case  $\lim_{t\rightarrow T}(T-t)^{\beta}u(x+y(T-t),t)=0$ in \eqref{3blowup} can be excluded. More precisely, we have the following result.
\begin{proposition}
Let $u_{0}$ be a nontrivial($\not\equiv 0$), nonnegative, radially symmetric function which is also nonincreasing in $|x|$ and satisfies
\begin{eqnarray}
\label{u0decay2}
|\nabla u_0|(x)\leq \frac{C}{1+|x|^{\delta}},
\end{eqnarray}
for some $\delta>0$.
 Let $u(x,t)$ be a finite time blow up solution of the equation \eqref{CP}, then
\begin{equation}\label{final1}
\lim_{t\rightarrow T}(T-t)^{\beta}u(0, t)>0.
\end{equation}
where $T$ is the finite blow up time in the sense of
\begin{eqnarray}
T:=\sup\big\{t>0 : \sup_{(x,t)\in \mathbb{R}^n\times(0,t)}u(x,t)<\infty\big\}.
\end{eqnarray}
\end{proposition}
\begin{proof}
Since $u_{0}$ is a nonnegative, radially symmetric function which is also nonincreasing in $|x|$ and satisfies \eqref{u0decay2}, it is easy to show that $u(x, t)$ is also nonnegative, radially symmetric function and  nonincreasing. Indeed, for any $e\in\mathbb{S}^n$, let $l$ be the plane of $e\cdot x=0$ which through the origin. For any $x\in\R^n$, denote $x'=x-2(e\cdot x)e$, the mirror image of $x$ among the plane $l$. Let $w(x,t)=u(x,t)-u(x',t)$, for any $T'<T$, we have
\begin{eqnarray*}
\left\{\begin{array}{l}
\partial_t w+(-\Delta)^{\frac{1}{2}}w +b(x,t)w=0, \ (x,t) \in \R^n\times(0,T']\\
w(x, 0)= 0, \ x\in \R^n,
\end{array}
\right.
\end{eqnarray*}
where $b(x,t)$ is bounded on $\R^n\times[0,T']$. Then the maximum principle (see lemma 4.1 in \cite{JX2}) tells us that $w\equiv0$, which means $u$ is symmetric among the plane $l$. Since the arbitrariness of $l$, we prove the radial symmetry of $u$. Next, we consider $h_j=\frac{\partial u}{\partial x_j}, j=1,\cdots, n$. For simplicity, we only consider $j=n$. By the radial symmetry of $u$, $h_n(\tilde{x},x_n,t)=-h_n(\tilde{x},-x_n,t)$ for all $(\tilde{x},x_n,t)\in\R^{n-1}\times\R\times[0,T']$. By \eqref{CPdecay}, we also have
\begin{eqnarray}
\lim_{|x|\to\infty} h_n(x,t)=0, \ \text{for any } \ t\in[0, T'].
\end{eqnarray}
Since $h_n$ satisfies the equation
\begin{eqnarray*}
\left\{\begin{array}{l}
\partial_t h_n+(-\Delta)^{\frac{1}{2}}h_n -p|u|^{p-1}h_n=0, \ (x,t) \in \R^n\times(0,T']\\
h_n(x, 0)\leq0, \ x\in \R^n_+,
\end{array}
\right.
\end{eqnarray*}
where $\R^n_+=\{(\tilde{x},x_n)\in\R^{n-1}\times\R : x_n>0\}$.  Then the maximum principle (see lemma 2.1 in \cite{JX2}) tells us that $h_n\leq0$ on $\R^n_+\times(0,T']$. In general, we have $h_j\leq0$ when $x_j>0$, $j=1,\cdots, n$, which means $u$ is decreasing in $|x|$.

By a contradiction argument, we assume
$$\lim_{t\rightarrow T}(T-t)^{\beta}u(0, t)=0,$$
then for every $\epsilon>0$, there exists a $\delta>0$ such that
$$u(x,t)<\epsilon(T-t)^{-\beta} \quad\text{on }\mathbb{R}^{n}\times(T-\delta, T).$$
Let $\hat{u}(x,t)=\delta^{\beta}u(\delta x, \delta (t+T))$, then $\hat{u}(x, t)$ satisfies
\begin{equation}
\left\{\begin{array}{lll}
\hat{u}_{t}+(-\Delta)^{\frac{1}{2}}\hat{u}-\hat{u}^{p}=0\quad\text{in }\mathbb{R}^{n}\times(-1, 0)\\
\hat{u}(x,t)<\epsilon (-t)^{-\beta}\quad\text{in }\mathbb{R}^{n}\times(-1, 0).
\end{array}
\right.
\end{equation}
The semigroup representation formula for $\hat{u}(x, t)$ gives
\begin{equation}\label{u(x,t)}
\hat{u}(x,t)=\int_{\mathbb{R}^{n}}P(x-y, t+1)\hat{u}(y, -1)dy+\int_{-1}^{t}\int_{\mathbb{R}^{n}}P(x-y, t-s)\hat{u}^p(y,s)dyds.
\end{equation}
Here $P(x,t)$ is the fractional heat kernel defined by \eqref{halfheatkernel}.
By \eqref{u(x,t)}, we know that
\begin{equation}\label{u(x,t)e}
\|\hat{u}(\cdot, t)\|_{L^{\infty}(\mathbb{R}^{n})}\leq \epsilon+\int_{-1}^{t}\epsilon^{p-1}s^{-1}\|\hat{u}(\cdot, t)\|_{L^{\infty}(\mathbb{R}^{n})}ds,
\end{equation}
By Gronwall's inequality, we can get that
\begin{equation}\label{u(x,t)e1}
\begin{aligned}
\|\hat{u}(\cdot, t)\|_{L^{\infty}(\mathbb{R}^{n})}&\leq \epsilon\exp{\int_{-1}^{t}\epsilon^{p-1}s^{-1}ds}\\
&\leq2\epsilon(-t)^{-\epsilon^{p-1}}.
\end{aligned}
\end{equation}
We take \eqref{u(x,t)e1} into \eqref{u(x,t)}, then
\begin{equation}
\|\hat{u}(\cdot, t)\|_{L^{\infty}(\mathbb{R}^{n})}\leq \epsilon+\int_{-1}^{t}(2\epsilon)^{p-1}(-s)^{-(p-1)\epsilon^{p-1}}\|\hat{u}(\cdot, t)\|_{L^{\infty}(\mathbb{R}^{n})}ds.
\end{equation}
Let $\epsilon$ be a constant which is small enough. If we apply Gronwall's inequality again, then we can get that $\hat{u}$ is bounded. It follows that $u(x,t)$ is bounded near $t=T$. Since we have assumed that $u(x,t)$ is a finite time blow up solution, this is a contradiction.
\end{proof}

\section{Appendix: Computation of $c_1M_1$}
For $n=1$, $f_j(y), j=1,2$ defined by \eqref{f12} has an explicit expression. Indeed, recall that $\rho(y)=\frac{1}{1+y^2}$, then
\begin{eqnarray}\label{M11}
  \begin{aligned}
  f_1(y):=&\ \frac{1}{\rho(y)}
  \int_{ B_{|y|}(0)}\frac{(\rho(y')-\rho(y))^2}{|y'-y|^{2}}\frac{1}{\rho(y')}dy'
  \\=&\
  \frac{1}{1+y^2}\int_{ B_{|y|}(0)}\frac{(y'+y)^2}{1+y'^2}dy'
   \\=&\
  \frac{2}{1+y^2}\int_0^{|y|}dy'
  +\frac{2(y^2-1)}{1+y^2}\int_0^{|y|}\frac{dy'}{1+y'^2}
  \\=&\
  \frac{2|y|}{1+y^2}+\frac{2(y^2-1)}{1+y^2}\arctan|y|.
  \end{aligned}
\end{eqnarray}
Similarly,
\begin{eqnarray}\label{M12}
 \begin{aligned}
  f_2(y):=&\ \frac{1}{\rho(y)^2}
  \int_{\mathbb{R}\setminus B_{|y|}(0)}\frac{(\rho(y')-\rho(y))^2}{|y'-y|^{2}}dy'
  \\=&\
  \int_{\mathbb{R}\setminus B_{|y|}(0)}\frac{(y'+y)^2}{(1+y'^2)^2}dy'
  \\=&\
  2\int_{|y|}^\infty\frac{dy'}{1+y'^2}
  +2(y^2-1)\int_{|y|}^\infty\frac{dy'}{(1+y'^2)^2}
   \\=&\
  \pi-2\arctan|y|
  +(y^2-1)\big(\frac{\pi}{2}-\arctan|y|
  -\frac{|y|}{1+y^2}\big).
  \end{aligned}
\end{eqnarray}
Since $c_1=\frac{1}{\pi}$, we are going to prove that $M_1<4\pi$. Since $f_j(y), j=1,2$ are even, we may
assume $y\in[0,+\infty)$. It is not hard to see that
\begin{eqnarray}
 \begin{aligned}
  \frac{2|y|}{1+y^2}\leq&\ 1,\\
  \frac{2(y^2-1)}{1+y^2}\arctan|y|\leq&\ \pi,\\
  \pi-2\arctan|y|\leq&\ \pi.
 \end{aligned}
\end{eqnarray}
Let
\begin{eqnarray}
  f(y)=(y^2-1)\big(\frac{\pi}{2}-\arctan|y|-\frac{|y|}{1+y^2}\big),
\end{eqnarray}
then
\begin{eqnarray}
  f'(y)=2y\big(\frac{\pi}{2}-\arctan|y|-\frac{|y|}{1+y^2}\big)
  -\frac{2(y^2-1)}{(1+y^2)^2}.
\end{eqnarray}
Observe
\begin{eqnarray}
  f(y)<0=f(1),\ \text{for all}\ y\in[0,1),
\end{eqnarray}
and, by the L'H\^opital's rule,
\begin{eqnarray}
  \begin{aligned}
  \lim_{y\rightarrow\infty}f(y)=&\
  \lim_{y\rightarrow\infty}
  \frac{\frac{2}{(1+y^2)^2}}{\frac{2y}{(y^2-1)^2}}
  \\=&\ \lim_{y\rightarrow\infty}\frac{2}{y}=0.
  \end{aligned}
\end{eqnarray}
Then $f$ achieves its maximum at some critical point $y_1\geq1$. $f'(y_1)=0$ implies that
\begin{eqnarray}
  \frac{\pi}{2}-\arctan y_1-\frac{y_1}{1+y_1^2}=\frac{y_1^2-1}{y_1(1+y_1^2)^2}.
\end{eqnarray}
It follows that
\begin{eqnarray}
  f(y_1)=\frac{(y_1^2-1)^2}{y_1(1+y_1^2)^2}<\frac{1}{y_1}\leq1.
\end{eqnarray}
Therefore, we conclude
\begin{eqnarray}
  M_1\leq1+\pi+\pi+1<4\pi.
\end{eqnarray}
In fact, a  numerical calculation shows that $M_1\approx4.8271<4\pi$. As a consequence, $p_*(1)\approx4.2072$.

\section*{Acknowledgements} The research of J. Wei is partially supported by  NSERC of Canada. The research of B. Deng and K. Wu is supported by China Scholar Council. The research of B. Deng is also
supported by Natural Science Foundation of China (No.1172110 and  No.11971137).
We would like to thank D. Gomez for some technical support for numerical computation. We also thank H. Zaag
for pointing out a mistaken statement.

\end{document}